\documentclass[11pt,a4paper]{article}
\usepackage[linesnumbered,ruled]{algorithm2e}
\usepackage{titlesec}
\SetAlCapNameFnt{\small }
\SetAlCapFnt{\small }
\usepackage[title]{appendix}
\usepackage{soul}
\usepackage{amsmath,amssymb,amsbsy,amsthm,calc,cases,enumerate,url,framed,fancyhdr,lipsum,tikz,enumitem,natbib,etoolbox,bbm}
\usepackage{graphicx}
\usepackage[pageanchor,hypertexnames=false]{hyperref}
\usepackage{minitoc}
\usepackage{multirow} 
\usepackage{tabularx}

\newcommand{\unparskip}{\vspace{-\parskip}}

\titlespacing\section{0pt}{12pt plus 4pt minus 2pt}{0pt plus 2pt minus 2pt}
\titlespacing\subsection{0pt}{12pt plus 4pt minus 2pt}{0pt plus 2pt minus 2pt}
\titlespacing\subsubsection{0pt}{12pt plus 4pt minus 2pt}{0pt plus 2pt minus 2pt}

\newtheorem{theorem}{Theorem}
\newtheorem{corollary}{Corollary}
\newtheorem{lemma}{Lemma}%
\usepackage{bbm}
\usepackage{stmaryrd}
\newtheorem{proposition}{Proposition}%
\newtheorem{assumption}{Assumption}%

\BeforeBeginEnvironment{proof}{\unparskip}
\BeforeBeginEnvironment{proof_outline}{\unparskip}
  
\BeforeBeginEnvironment{theorem}{\vspace{0.5\parskip}}
\BeforeBeginEnvironment{corollary}{\vspace{0.5\parskip}}
\BeforeBeginEnvironment{lemma}{\vspace{0.5\parskip}}
\BeforeBeginEnvironment{proposition}{\vspace{0.5\parskip}}
  
\BeforeBeginEnvironment{thm}{\vspace{0.5\parskip}}
\BeforeBeginEnvironment{cor}{\vspace{0.5\parskip}}
\BeforeBeginEnvironment{prop}{\vspace{0.5\parskip}}
\BeforeBeginEnvironment{lem}{\vspace{0.5\parskip}}

\newcommand{\iid}{\overset{\mathrm{iid}}{\sim}}

\newcommand{\R}{\mathbb{R}}

\newcommand{\dotcup}

\usepackage{xr}

\theoremstyle{remark}
\newtheorem{example}{Example}%
\theoremstyle{remark}
\usepackage[normalem]{ulem}

\begin{document}

\title{Asymptotics for parametric martingale posteriors}

\author{Edwin Fong$^{1,}$\thanks{Corresponding author. Email: chefong@hku.hk} \ and Andrew Yiu$^{2}$  \\ \\
$^1$Department of Statistics and Actuarial Science, University of Hong Kong\\
$^2$Department of Statistics, University of Oxford
}
\date{}

\maketitle
\vspace{-5mm}

\begin{abstract}
The martingale posterior framework is a generalization of Bayesian inference where one elicits a sequence of one-step ahead predictive densities instead of the likelihood and prior. 
Posterior sampling then involves the imputation of unseen observables, and can then be carried out in an expedient and parallelizable manner using {predictive resampling} without requiring Markov chain Monte Carlo. Recent work has investigated the use of plug-in parametric predictive densities, combined with stochastic gradient descent, to specify a parametric martingale posterior. This paper investigates the asymptotic properties of this class of parametric martingale posteriors. In particular, two central limit theorems based on martingale limit theory are introduced and applied. The first is a {predictive} central limit theorem, which enables a significant acceleration of the predictive resampling scheme through a hybrid sampling algorithm based on a normal approximation. The second is a Bernstein-von Mises result, which is novel for martingale posteriors, and provides methodological guidance on attaining desirable frequentist properties.
We demonstrate the utility of the theoretical results in simulations and a real data example. 
\end{abstract}


\section{Introduction}

\subsection{Martingale posteriors}
The martingale posterior \citep{Fong2023} is a generalization of the Bayesian posterior which places {prediction} at the heart of Bayesian inference. The key insight is the equivalence between posterior uncertainty on the parameter of interest and the predictive uncertainty of yet-to-be seen observables, which allows the elicitation of a sequence of one step ahead predictive densities to replace the usual likelihood and prior construction of a Bayesian model. 
Given observations $Y_{1:n}$, posterior computation now involves the following sequential imputation
\begin{align*}
    Y_{n+1} \sim p(\cdot \mid Y_{1:n}), \quad Y_{n+2}\sim p(\cdot \mid Y_{1:n+1}), \quad \ldots  \quad Y_{N} \sim p(\cdot \mid Y_{n+1:N-1}),
\end{align*}
taking $N \to \infty$, where $p(\cdot \mid Y_{1:i})$ is the one step ahead predictive density conditional on the first $i$ observations. The martingale posterior is then the distribution of the parameter of interest computed as a functional of the limiting imputed population, which we write informally as $\theta(Y_{1:\infty})$. 

One key benefit of the martingale posterior is the ability to work with a much larger class of models for Bayesian inference. These include conditionally identically distributed sequences \citep{Berti2004, Fong2023}, quantile estimates \citep{Fong2024}, and plug-in parametric predictive densities \citep{Walker2022,Holmes2023}, the last of which forms the focus of this paper. Surprisingly, the sequence of predictives may often be elicited in a completely prior-free manner, allowing for noninformative Bayesian inference without a prior distribution. 
In addition to gains in modelling flexibility, posterior computation is also drastically different and allows for significant gains in computation speed. The sequential imputation scheme, also known as {predictive resampling}, relies on recursive simulation and updating of the predictive density with a suitable truncation point $N$, removing the need for Markov chain Monte Carlo. This sampling scheme can also be entirely parallelized and utilize modern graphics processing units, and is free of the convergence challenges faced by Markov chain Monte Carlo.

This paper focuses on developing asymptotic theory for the parametric martingale posterior.  In particular, we will show two central limit theorems under different asymptotic regimes. The first will involve predictive asymptotics, where we take the imputed population size to infinity. The second will consider frequentist asymptotics, where the number of independent and identically distributed observations grows to infinity, as in the Bernstein-von Mises theorem. We will show that the theoretical results will have practical impacts on accelerating posterior computation and on attaining frequentist coverage.

\subsection{Related works}
There has been a recent resurgence in direct focus on the predictive distribution within Bayesian inference, including works such as \cite{Fortini2020,Berti2021, Fong2023} and references within.
The parametric Bayesian bootstrap is introduced and explored in 
\cite{Walker2022,Holmes2023, Wang2024}, with an extension to mixture models in \cite{Cui2023}.  Recently, \cite{Fortini2024} also considers the parametric martingale posterior for logistic regression, and \cite{Garelli2024} considers the predictive asymptotics of  parametric predictives based on a mean and variance estimate.
Our work unifies the examples in these works, and we will provide a concrete suggestion on an appropriate learning rule for good frequentist behaviour, which is missing from the literature. 

Predictive central limit theorems have been previously explored in the literature by works such as \cite{Fortini2020,Berti2021,Fortini2023}. In particular, \cite{Fortini2020} and \cite{Favaro2024} utilize a result of \cite{Crimaldi2009}
to construct asymptotic credible intervals for Newton's algorithm, which can also be framed as a nonparametric martingale posterior.
\cite{Fong2024} also utilizes a central limit theorem to carry out approximate posterior sampling from a nonparametric quantile martingale posterior.
Closest to our first result to come is the predictive central limit theorem for the logistic regression example provided in \cite{Fortini2024}, which we will extend to more general settings shortly. More generally, the parametric martingale posterior is closely connected to the parametric bootstrap \citep{Efron1971}, the Bayesian bootstrap \citep{Rubin1981}, as well as generalized Bayesian methods \citep{Bissiri2016, Knoblauch2022}.

\section{Methodology}

In this section, we introduce a specific form of the parametric martingale posterior distribution, where the sequence of predictive densities is parametric. Predictive sequences of this type have been investigated in \cite{Walker2022,Holmes2023,Fortini2024, Garelli2024}, and the predictive sequence below encompasses them in a sense. Let $\{p_\theta(y): \theta \in \Theta\}$ denote a family of probability density functions on $y \in \mathcal{Y}$, which will serve as our predictive density. A classic example is the normal family with fixed variance, where $p_\theta(y) = \mathcal{N}(y \mid \theta,\sigma^2)$, where $\mathcal{Y} = \R$, $\Theta = \R$ and $\sigma^2$ is a fixed constant. For ease of exposition, we begin with assuming $\Theta \subseteq \R$, with the trivial extension to the multivariate case provided in Section \ref{sec:mv} and Section \ref{sec:app_mv} of the Appendix. For the remainder of the paper, we make the following assumptions on $\log p_\theta(y)$, which is the score function of the model. 
\begin{assumption}\label{as:existence}
For each $\theta \in \Theta \subseteq \R$, we require the following, where expectations are understood to be over $Y \sim p_\theta$.
The score function $s(\theta,y)= \partial\log p_\theta(y) /\partial \theta$ exists, is finite, and has mean $0$, that is 
    $
    E_\theta\left\{s(\theta,Y) \right\} = 0.
    $
Furthermore, the score function is square-integrable, that is 
    $\mathcal{I}(\theta) =  E_\theta\left\{s(\theta,Y)^2\right\} < \infty,$
   where $\mathcal{I}(\theta)$ is the Fisher information. Finally, we require that $\mathcal{I}(\cdot)$ is continuous at $\theta$ and satisfies  $\mathcal{I}(\theta)> 0$.
\end{assumption}
Suppose we have observed $Y_{1:n}$ drawn from $p_{\theta^*}$, where $\theta^* \in \R$ is some true underlying value. We will discuss this assumption of model well-specification in a later section. Now consider the plug-in predictive density $p_n(y) = p_{\theta_n}(y)$, which will act as our starting point for the martingale posterior. Here, $\theta_n$ is an estimate of $\theta$ computed from $Y_{1:n}$, e.g. the maximum likelihood estimate. We now introduce the recursive update which will let us impute the remainder of the population.

Consider the random sequence of predictive densities, $\left(p_{\theta_N}(y)\right)_{N = n+1,n+2,\ldots}$ where $Y_{N+1} \sim p_{\theta_N}(\cdot)$ and the parameter estimate satisfies
\begin{equation}\label{eq:sgd}
\theta_{N} = \theta_{N-1} + N^{-1} \,\mathcal{I}(\theta_{N-1})^{-1} \, s(\theta_{N-1},Y_{N})
\end{equation}
for all $N > n$. Crucially, the learning rate of $N^{-1}$ ensures we have
\begin{align}\label{eq:alpha}
\sum_{N = n+1}^\infty  N^{-1}= \infty, \quad \sum_{N= n+1}^\infty N^{-2} < \infty,
\end{align}
as is standard in stochastic approximation. All of the results in this paper hold if we replace $N^{-1}$ with a general sequence $\alpha_N$, as long as the above is satisfied. The update (\ref{eq:sgd}) is  considered for logistic regression in \cite{Fortini2024} without the Fisher information and in the general case in \cite{Walker2022,Holmes2023,Wang2024}.

The above can be interpreted as the online learning rule for the sequence of the predictive densities and has a clear connection to stochastic approximation \citep{Lai2003}, also referred to as stochastic gradient descent.
 We can thus leverage tools from the rich stochastic approximation literature to help study the implied martingale posterior. The key difference with the traditional stochastic approximation setup is that the data $(Y_{N})_{N >n}$ is imputed, instead of being independently and identically distributed from $p_{\theta^*}$. 
The preconditioning of the gradient by the inverse Fisher information is also akin to the Hessian matrix used in Newton's method for optimization, and is also termed the natural gradient in the machine learning literature \citep{Martens2020}.
This deliberate choice will have ramifications for the asymptotic variance later.

Define the filtration $\mathcal{F}_N = \sigma(Y_{n+1},\ldots,Y_{N})$, where $Y_{1:n}$ is considered fixed for now. It is clear that $\theta_N$ is a martingale, as we have
$E\left\{s(\theta_N,Y_{N+1})\mid \mathcal{F}_N\right\} = 0$
by assumption. Under appropriate assumptions, which will be investigated in detail in the next section, we then have that $\sup_N E\left(|\theta_N|\right)< \infty$, so $\theta_N$ is a martingale adapted to $\mathcal{F}_N$ and bounded in $L_1$. From Doob's martingale convergence theorem \citep{Doob1953}, there exists a finite $\theta_\infty \in \R$ such that
$\theta_N \to \theta_\infty$
almost surely. The distribution of the random variable $\theta_\infty$ conditional on $Y_{1:n}$ is then termed the {parametric martingale posterior}.

Up until now, we have omitted discussion on the initial estimate $\theta_n$. 
As this work will investigate the frequentist properties of the parametric martingale posterior, the results will depend on the chosen estimate $\theta_n$.
In particular, we could enforce the same update (\ref{eq:sgd})
for $N = 1,\ldots, n$ where $\theta_0 \in \R$ is an initial value, which is a form of coherency as the estimation matches predictive resampling. 
As discussed in \cite{Walker2022}, one can then view Bayesian inference as applying update (\ref{eq:sgd}) for independent and identically distributed observations $Y_{1:n}$, then continuing the same learning scheme with imputed data $Y_{n+1:\infty}$\, from $p_{\theta_N}$ once we are depleted of real observations. 
Alternatively, it may be simpler and acceptable to utilize a well-understood estimator such as the maximum likelihood estimate for $\theta_n$, as long as it is asymptotically equivalent to using (\ref{eq:sgd}).

\section{Predictive central limit theorem}
\subsection{Predictive asymptotics}
We begin our study with consideration of {predictive} asymptotics, which we distinguish from {frequentist} asymptotics. In this context, we treat $Y_{1:n}$ as fixed, and consider the convergence of $\theta_N$ to $\theta_\infty$
as $N \to \infty$, which was also considered by \cite{Doob1953}. In other words, we want to study predictive resampling as our imputed population $Y_{n+1:N}$ grows with $N \to \infty$. We will dedicate some time to emphasize this rather unique type of asymptotics that is inherently Bayesian, and this type of asymptotics is only made explicit through the predictive view point.

In this section, we will consider the limiting distribution of $\left({\theta_\infty - \theta_N}\right)$, scaled appropriately,
as $N \to \infty$, utilizing martingale limit theory \citep{Hall2014}. Intuitively, the above will quantify how much uncertainty will remain, as well as the distribution of this uncertainty, when we only see a finite $N$ members of the population, instead of the full infinite population. Practically, one can also imagine studying the convergence rate of the predictive resampling scheme as $\theta_N \to \theta_\infty$ almost surely. However, we must emphasize that $\theta_\infty$ is {random}, so we are considering convergence of a random object $\theta_N$ to another random object $\theta_\infty$. An unintuitive result will be that the limiting law of $\theta_\infty$ will have random components, which is in stark contrast to the frequentist asymptotics. As a result, we argue that predictive asymptotics are only helpful for the {computational} aspect of the parametric martingale posterior, whereas a Bernstein-von Mises type result (as in Section \ref{sec:bvm}) is of greater statistical relevance.

Predictive asymptotics have previously been studied in \cite{Fortini2020, Fortini2024}, and they provide insightful interpretations of the above type of asymptotics. We consider our results as a generalization of theirs, with a tailoring to the parametric case. Interestingly, the assumptions are somewhat more tedious in the parametric case, as $\theta \in \Theta \subseteq \R$ may not have compact support.
This is in contrast to the nonparametric case, where the martingale constructed is usually the cumulative distribution function, which is bounded by 1. 
 Another  novelty we demonstrate is that quantification of the above distribution will actually allow us to justify a computational scheme which drastically accelerates posterior sampling, which will have strong practical benefits.

\subsection{Martingale central limit theorem}
 To find the limiting distribution under predictive resampling, we now utilize the seminal work in martingale central limit theory \citep{Hall2014}, which draws strong analogies between independent and identically distributed variables to martingales.  However, while elegant, the assumptions for martingale central limit theorems can become quite technical, which is especially true in our case as our estimates $\theta_N$ may not lie in a compact set.  A discussion of this is also provided in \cite{Walker2022}.

For martingale central limit theorems to hold, we require that our martingale is bounded in $L_2$, that is
$\sup_N E\left(\theta_N^2\right) < \infty$. This is stronger than the bounded in $L_1$ condition required for the existence of $\theta_\infty$, and is equivalent to the bounded variance condition in traditional central limit theorems. However, as we will see shortly, a stronger Lindeberg condition is actually needed for a martingale central limit theorem to hold. To that end, we introduce the following set of assumptions which help control the martingale under predictive resampling.

\begin{assumption}\label{as:UI}
Let $Z_N$ denote the natural gradient, that is
\begin{align}\label{eq:Z}
    Z_N = \mathcal{I}(\theta_{N-1})^{-1}\, s(\theta_{N-1}, Y_N).
\end{align}
The sequence $\left(Z_N^2\right)_{N > n}$ is uniformly integrable under predictive resampling with (\ref{eq:sgd}) starting from the initial $\theta_n$.
\end{assumption}

The above object $Z_N$ is sometimes termed the natural gradient, and is also the efficient influence function for $\theta$ under well-specification of the parametric model \citep{van2000},  which we require to be square uniformly integrable. 
The above condition is then sufficient for a martingale version of Lindeberg's condition to hold. While the condition is easy to state, it may be challenging to check, which is typical of assumptions required for martingale central limit theorems \citep{Heyde1979}. This difficulty arises due to the fact that we are converging to a random $\theta_\infty$, which can lie anywhere in $\R$. As a result, we cannot always bound $Z_N$, where boundedness was previously utilized in the results of \cite{Fortini2020, Fortini2024}. Fortunately, we will introduce an easy to check assumption which holds for many choices of $p_\theta(y)$ as long as the learning rate satisfies (\ref{eq:alpha}). We postpone the verification of Assumption \ref{as:UI} until Section \ref{sec:UI}.

 We now show that Assumption \ref{as:UI} implies another key condition and that the martingale is bounded in $L_2$, which is required for the central limit theorem.
\begin{lemma}\label{lem:L2}
   Under Assumption \ref{as:UI}, we have that $\left(\mathcal{I}(\theta_{N-1})^{-1}\right)_{N > n}$ is uniformly integrable. Furthermore, $\theta_N$ is a martingale bounded in $L_2$, so  $\theta_N \to \theta_\infty$ almost surely.
\end{lemma}
We are now equipped to state the main theorem of this section.
\begin{theorem}\label{th:pred_CLT}
For $X_N = \theta_N - \theta_{N-1}$, we define
$$
V_N^2 = \sum_{i=N+1}^\infty E\left(X_i^2 \mid \mathcal{F}_{i-1}\right), \quad s_N^2 = \sum_{i=N+1}^\infty E\left(X_i^2\right) = E\left\{(\theta_\infty - \theta_N)^2\right\}.
$$
Under Assumptions \ref{as:existence} and \ref{as:UI}, we have
$$
\frac{V_N^2}{s_N^2} \rightarrow \frac{\mathcal{I}(\theta_\infty)^{-1}}{E\left\{\mathcal{I}(\theta_\infty)^{-1}\right\}}
$$
almost surely, and 
\begin{equation}\label{eq:pred_CLT}
    V_N^{-1} \left(\theta_\infty - \theta_N \right) \to\mathcal{N}(0,1)
\end{equation}
in distribution.
\end{theorem}
\begin{proof}
We outline the proof, with details deferred to Section \ref{sec:app_predclt} of the Appendix.
   The proof is an application of a variant of the tail sum martingale central limit theorem, e.g. \citet[Corollary 3.5]{Hall2014} or \citet[Exercise 6.7]{Hausler2015}.   To summarize, Assumption \ref{as:UI} implies a conditional Lindeberg condition, and furthermore implies that the normalized variance $V_N^2/s_N^2$ converges almost surely. These two conditions are then sufficient for the martingale central limit theorem for tail sums to hold. 
\end{proof}
Many variants of the martingale central limit theorem exist, which we specialize for our use cases in Sections \ref{sec:app_tail} and \ref{sec:app_infmart} in the Appendix.  A key component of the above result is that the scaled limit of the normalizing variance $V_N^2$ is {random} and equal to $\mathcal{I}(\theta_\infty)^{-1}$, so the variance of the limiting distribution of $s_N^{-1}(\theta_\infty - \theta_N)$ is also random. Here, $s_N^{-1}$ is the deterministic rate, whereas the ratio $V_N^2/s_N^2$ plays the role of the random variance. For large $N$, the heuristic argument is
$$
V_N^2 \approx \mathcal{I}(\theta_\infty)^{-1} \sum_{i=N}^\infty i^{-2}, \quad s_N^2 \approx E\left\{\mathcal{I}(\theta_\infty)^{-1}\right\}\sum_{i=N}^\infty i^{-2},
$$
where $\mathcal{I}(\theta_\infty)^{-1}$ is random but $E\left\{\mathcal{I}(\theta_\infty)^{-1}\right\}$ is constant.
We will have $s_N^{-1}\propto {N}^{1/2}$, which is the usual parametric rate.

The random limit of $V_N$ is in stark contrast to classical central limit theorems, where the normalizing variance will usually converge to a constant.
The ability to normalize $(\theta_\infty - \theta_N)$ by $V_N$ and thus write the limit distribution in (\ref{eq:pred_CLT}) independently of $\theta_\infty$ is crucial for the practical use of Theorem \ref{th:pred_CLT} for approximate posterior sampling, as we will see shortly.
The key to this form of (\ref{eq:pred_CLT}) is precisely the stable convergence of martingale central limit theorems \citep[Chapter 3]{Hall2014} in the sense of \cite{Renyi1963}, where stable convergence is stronger than convergence in distribution and weaker than convergence in probability. The importance of stable convergence is also emphasized in \cite{Fortini2020,Fortini2024}. In the interest of space, we provide an extended discussion on stable convergence in Section \ref{sec:stab} of the Appendix, where it is needed for some of the technical proofs.

\subsection{Uniform integrability condition}\label{sec:UI}
Assumption \ref{as:UI} is unwieldy to work with in practice, as it is a condition on the entire predictive sampling procedure for $ N > n$.
Fortunately, we have the following helpful result which 
allows us to check Assumption \ref{as:UI} by only considering a single time step.
\begin{proposition}\label{prop:Z4}
    Let $Z(\theta,Y) = \mathcal{I}(\theta)^{-1}s(\theta, Y)$, where $Y \sim p_\theta$. Suppose there exists non-negative constants $B,C< \infty$ such that the following holds for all $\theta \in \Theta \subseteq \R$,
    \begin{align*}
    E_\theta\left\{Z(\theta,Y)^4\right\} \leq B + C \theta^4.
    \end{align*}
   Then Assumption \ref{as:UI} is satisfied. 
\end{proposition}
The intuition is that the above in combination with (\ref{eq:alpha}) implies $\sup_{N> n} E\left(Z_N^4\right) < \infty$, which is sufficient for the uniform integrability of $\left(Z_N^2\right)_{N >n}$. 
Although it is likely that bounding $E_\theta\left\{Z(\theta,Y)^{2+ \delta}\right\}$ for some $\delta > 0$  suffices, choosing $\delta = 2$ is particularly straightforward due to the recursive form of $\theta_N$. 
The above is relatively easy to check and is only dependent on the predictive model $p_\theta$; we now demonstrate this in the next examples.
\begin{example}\label{ex:UI_assump}
    Here we verify Assumption \ref{as:UI} when $Z(\theta,Y)$ is not bounded. Suppose our parametric model is the normal model with mean zero but unknown variance 
$p_\theta(y) = \mathcal{N}(0,\theta)$,
where $\theta \in \R^+$. The score function and Fisher information are respectively
\begin{equation*}
\begin{aligned}
s(\theta,y) &=  \frac{y^2 - \theta}{2\theta^2}, \quad 
\mathcal{I}(\theta) 
&= \frac{1}{2\theta^2}\cdot
\end{aligned} 
\end{equation*}
We then have $\mathcal{I}(\theta)^{-1}$ continuous in $\theta$ as needed. We highlight that $\mathcal{I}(\theta), \mathcal{I}(\theta)^{-1}$ and $Z(\theta,Y)$ are all unbounded, which is the cause of our difficulties. We can utilize Proposition \ref{prop:Z4} by computing $E\left\{Z(\theta,Y)^4\right\}$, where here we have
$Z(\theta,Y)^4=  \left(Y^2 - \theta\right)^4$.
Since $Y^2/\theta \sim \chi^2(1)$, which has a 4th central moment equal to 60, we have
\begin{align*}
E_\theta\left\{\left(Y^2 - \theta\right)^4\right\}&= 60\, \theta^4
\end{align*}
for all $\theta \in \R^+$. We can thus apply Proposition \ref{prop:Z4} with $B = 0$ and $C = 60$.\vspace{5mm}
\end{example}

\begin{example}\label{ex:t}
We now verify Assumption \ref{as:UI} for an example that is not in the exponential family, and in particular is heavy-tailed. Consider a location-scale Student-$t$ distribution with location $\theta$, scale $\tau$, and degrees of freedom $\nu$, which has the density
\begin{align*}
    p_\theta(y) = \frac{\Gamma\left(\frac{\nu + 1}{2}\right)}{\Gamma\left(\frac{\nu}{2}\right)\tau\,(\pi \nu)^{{1}/{2}} } \left\{1 + \frac{1}{\nu}\left(\frac{y -\theta}{\tau}\right)^2\right\}^{-\frac{\nu + 1}{2}}.
\end{align*}
For simplicity we assume $\tau = 1$. The update term can then be computed as
\begin{align*}
    Z(\theta,Y) =\frac{(\nu + 3)(Y-\theta)}{\nu + (Y-\theta)^2}
\end{align*}
which can be done using standard results, e.g. \cite{Lange1989}. We can then compute 
\begin{align*}
    Z(\theta,Y)^4 = \frac{(\nu+3)^4(Y-\theta)^4}{\left\{\nu + (Y- \theta)^2\right\}^4} \leq \frac{(\nu + 3)^4}{16\nu^2}
\end{align*}
where we have used the fact that $x^2/(\nu + x^2)^2 \leq (4\nu)^{-1}$ for $\nu > 0$.
Here, there is no dependence on $\theta$, so we can immediately apply Proposition \ref{prop:Z4} with $C = 0$. 
\end{example}

\subsection{Accelerating predictive resampling}
Predictive resampling with (\ref{eq:sgd}) requires the truncation at some finite $N$ in order for feasible computation, and while convergence appears relatively quick, the computational cost still grows linearly with $N$. 
The key idea of this section is that the form of Theorem \ref{th:pred_CLT}
allows us to approximate
\begin{align*}
    \theta_\infty \approx \theta_N + V_N\varepsilon,
\end{align*}
where $\varepsilon \sim \mathcal{N}(0,1)$, which will hopefully allow us to attain accurate sampling for a smaller value of $N$ and also approximate the gap between $\theta_\infty$ and $\theta_N$.

However, a minor difficulty is that we need a way of approximating the {random} variance $V_N^2$, which depends on the future, yet-to-be sampled values of $Y_{N+1:\infty}$. As the point is to avoid sampling $Y_{N+1:\infty}$ to save computation, we require an approximation, which is also discussed in \cite{Fortini2020}.  A suitable approximation is
\begin{align}\label{eq:pred_as_var}
    \widehat{V}_N^2 = \mathcal{I}(\theta_N)^{-1} \sum_{i = N}^\infty i^{-2},
\end{align}
where we have assumed that $\mathcal{I}(\theta_N)^{-1}$ is sufficiently close to $\mathcal{I}(\theta_\infty)^{-1}$, which seems reasonable given $\mathcal{I}(\theta)^{-1}$ is continuous.  As $\mathcal{I}(\theta_N)^{-1}$ is already being computed as the gradient preconditioner, there is no additional computation required. The following proposition justifies the above usage.
\begin{corollary}\label{corr:approx_variance}
    Under Assumptions \ref{as:existence} and \ref{as:UI}, we have that
    \begin{align*}
        \widehat{V}_N^{-1}\left(\theta_\infty - \theta_N\right)\to\mathcal{N}(0,1)
    \end{align*}
    in distribution.
\end{corollary}
\begin{proof}
    As $\mathcal{I}(\theta)^{-1}$ is continuous, we have $\mathcal{I}(\theta_N)^{-1} \to \mathcal{I}(\theta_\infty)^{-1}$ almost surely, so
   ${V_N}/{\widehat{V}_N} \to 1$
    almost surely. Applying Slutsky's theorem to  Theorem \ref{th:pred_CLT} then gives us the result.
\end{proof}

 \cite{Fortini2020,Fortini2024,Fong2024} suggest to directly use the initial estimate with $\theta_n + \widehat{V}_n\varepsilon$ as the approximation to the martingale posterior, where $n$ is the number of real observations. However, we highlight that this may inaccurate if $n$ is too small. Intuitively, for small $n$, the martingale posterior may quite far from a Gaussian distribution, and we cannot generate more real observations. The key idea behind our method is that we are free to impute as many simulated data points as desired given a computational budget, so we can just impute $Y_{n+1:N}$ until we reach a point where the Gaussian approximation is indeed accurate. We thus recommend a hybrid scheme, where we carry out regular predictive resampling for a moderate value of $N$ within the computation budget, then finish-off predictive resampling by adding $\widehat{V}_N \varepsilon$. This is illustrated in Algorithm \ref{alg:hybrid} below.

 \begin{figure}[!h]
\center
\begin{minipage}{.5\linewidth}
\begin{algorithm}[H]
\caption{Hybrid predictive resampling for a single posterior sample.}\label{alg:hybrid}
 Estimate $\theta_n$ from $Y_{1:n}$\\
 \For{$i \gets n+1$ \textnormal{\textbf{to}} $N$}{
  Draw $Y_{i} \sim p_{\theta_{i-1}} $ \\
   $\theta_i \leftarrow \theta_{i-1} +i^{-1} Z_i$ \\}
  Draw $\varepsilon \sim \mathcal{N}(0,1)$\\
 $\widehat{\theta}_\infty \leftarrow \theta_N + \widehat{V}_N \varepsilon$  \\
Output $\widehat{\theta}_\infty$
\end{algorithm}
\end{minipage}
\end{figure}

The key is that $N$ can be chosen to be much smaller than usual, as the remaining unsampled variation can be approximated by a Gaussian variable with random variance.
To choose a suitable $N$, one potential guideline is to choose $N$ such that $N+n = K p$ where $p$ is the dimensionality of $\theta$, and $K$ is a chosen constant, for example $K = 100$. We can choose $K$ sufficiently large so the asymptotic approximation holds well.
We now show in a simulation example that the hybrid scheme will be much better than truncation or applying the Gaussian approximation directly.\vspace{5mm}

\begin{example}\label{ex:exp}
Consider the exponential distribution as the predictive $p_\theta(y) = \text{Exp}(\theta)$, where $\theta$ is the scale parameter. The estimate of $\theta_N$ is simply the mean, which has the recursive update
\begin{align*}
 \theta_N   = \theta_{N-1} + N^{-1}(Y_N - \theta_{N-1}).
\end{align*}
We can easily apply Proposition \ref{prop:Z4} and compute (\ref{eq:pred_as_var}); see Section \ref{sec:expon} of the Appendix. To demonstrate the benefits of the hybrid scheme, we simulate $n = 10$ data points from $\text{Exp}(\theta = 1)$, and compare the parametric martingale posterior obtained using the hybrid scheme with truncation at $N = 30$ to the ground truth which is estimated by truncating at a large value of $N= 20000$ (which is essentially exact numerically). As baselines, we compare to the direct Gaussian approximation of \cite{Fortini2020,Fortini2024} and the truncated parametric martingale posterior at $N = 30$ without the Gaussian correction. Figure \ref{fig:hybrid} illustrates the kernel density estimates of $B = 50000$ parametric martingale posterior samples for the various sampling schemes. We see that the hybrid and exact scheme are very close, while the truncated scheme is noticeably too narrow and the Gaussian is inaccurate due to the skewness of the  martingale posterior. In terms of computation time on an Apple M2 Pro chip, the exact sampler required 6.4s, whereas the truncated/hybrid  sampler required $6.3 \times 10^{-4}$s, indicating a massive speed-up roughly equal to the ratio of the truncation values. This example illustrates that the first few imputed samples account for most of the non-Gaussianity of the martingale posterior, resulting in an accurate approximation via the hybrid scheme.

\begin{figure}
\begin{center}
\includegraphics[width=0.65\textwidth]{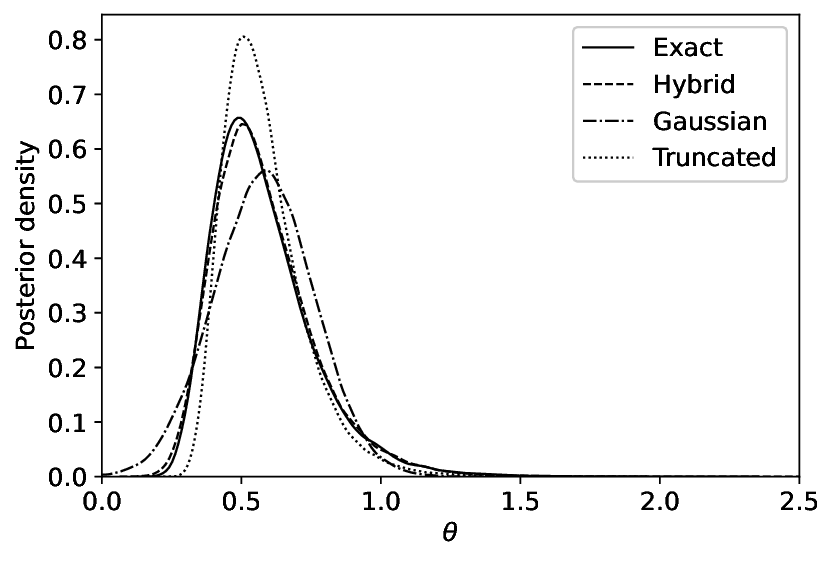}
\end{center}
\vspace{-7mm}
\caption{Kernel density estimate plots of $B = 50000$ parametric martingale posterior samples for the exact (solid), hybrid (dashes), truncated (dot) and Gaussian (dot-dash) sampling schemes.}
\label{fig:hybrid}
\end{figure}
 \end{example}

\section{Bernstein-von Mises theorem}\label{sec:bvm}
\subsection{Frequentist asymptotics}
In this section, we turn to a frequentist study of the parametric martingale posterior. We now have $Y_{1:n}$  independently and identically distributed according to $P^*$, where we will assume that the model is well-specified, i.e. there exists $\theta^*$ such that $p^* = p_{\theta^*}$. We consider model misspecification in Section \ref{sec:misspec}.
We would then like to study the martingale posterior  as $n \to \infty$.  We begin by first illustrating posterior consistency, before showing a Bernstein-von Mises theorem.

Again, we consider the update (\ref{eq:sgd}) for predictive resampling. The notation is slightly more involved due to the two data regimes for $Y_{1:n}$ and $\widetilde{Y}_{n+1:\infty}$, where we now use $\widetilde{Y}$ and double indexing to distinguish between simulated and observed data. The index $n$ corresponds to the observed samples, whilst $i$ indexes the imputed data. 
   Consider the parametric martingale posterior starting at $\theta_n$ with sample size $n$, that is for $i \geq 1$:
\begin{align*}
    \widetilde{Y}_{ni}  \sim p_{\theta_{n,i-1}}, \quad 
    \theta_{ni} = \theta_{n,i-1} + (n+i)^{-1} \mathcal{I}\left( \theta_{n,i-1}\right)^{-1} s\left(\theta_{n,i-1}, \widetilde{Y}_{ni}\right)
\end{align*}
 where $\theta_{n0} = \theta_n$ and $(\widetilde{Y}_{ni})_{i \geq 1}$ indicates the simulated population after observing $n$ observations from $P^*$. 
  As before, we will write the update function or natural gradient as $  Z_{ni} = \mathcal{I}\left( \theta_{n,i-1}\right)^{-1} s\left(\theta_{n,i-1}, \widetilde{Y}_{ni}\right)$ for shorthand. 
  
  As we are interested in the frequentist properties of the martingale posterior, we can first consider taking $i \to \infty$, where $\theta_{n\infty}$ corresponds to the almost sure limit of the martingale for each $n$, assuming we have the appropriate conditions for martingale convergence. We can then study the distribution of $\theta_{n\infty}$ as $n \to \infty$. To that end, we have the following assumption.
  
  \begin{assumption}\label{as:bvm}
  Suppose  $\mathcal{I}(\theta)^{-1}$ is continuously differentiable in a neighbourhood of $\theta^*$. 
  Furthermore, assume the conditions of Proposition \ref{prop:Z4} hold. 
  \end{assumption}
  
Intuitively, the first part of Assumption \ref{as:bvm} ensures that the variance of the update $Z_{ni}$ does not change too drastically with the initial $\theta_n$. As before, the second part is to prevent the variance of the martingale posterior from growing too quickly when predictive resampling. As we showed in Lemma \ref{lem:L2}, $\theta_{n\infty}$ exists almost surely as $\left(\theta_{ni}\right)_{i \geq 1}$ is a martingale in $L_2$ for each $n\geq 1$.
  
  In the proofs of the results to come, it is actually sufficient for the array $\left(Z_{ni}^2\right)_{i \geq 1, n \geq 1}$ to be uniformly integrable, but this is more arduous than Assumption \ref{as:UI} as it will also depend on the sequence $\theta_n$.
 We thus opt to work with the conditions of Proposition \ref{prop:Z4},  which is simpler to verify and implies the uniform integrability condition if $(\theta_n)_{n \geq 1}$ is convergent; see Proposition \ref{app:prop:Z4_bvm} in the Appendix for further details. In practice, checking Assumption \ref{as:bvm} is the same exercise as for the predictive asymptotic result, which we have demonstrated earlier.

\subsection{Posterior consistency}
In order to obtain posterior consistency, we merely require that the sequence $(\theta_n)_{n\geq 1}$ is a consistent estimator. For example, consistency of the stochastic approximation estimator $\theta_n$ via (\ref{eq:sgd}) is usually guaranteed by the below.
\begin{assumption}\label{as:consistency}
 The expected score function, $M(\theta) = \int  s(\theta,y) \, dP^*(y)$, exists and has a unique zero at $\theta^*$. Furthermore, $M(\theta) >0$ if $\theta  < \theta^*$ and $M(\theta) <0$ if $\theta  >\theta^*$. 
\end{assumption}
Under Assumption \ref{as:consistency} and additional regularity conditions, we can show that $\theta_n \to \theta^*$  $P^*$-almost surely using the almost supermartingale theorem of  \cite{Robbins1971}. Alternatively, we can let $\theta_n$ be a well-understood point estimator with good properties, such as the maximum likelihood estimator; see Section \ref{sec:init_estimate} in the Appendix for more details.

We follow a similar approach to \cite{Fong2024}  to study the consistency of the parametric martingale posterior based on Markov's inequality. 
\begin{proposition}
Suppose $\left(\theta_n\right)_{n \geq 1}$ is estimated from $Y_{1:n}$ drawn from $P^*$ and is strongly consistent at $\theta^*$.
Under Assumptions \ref{as:existence} and \ref{as:bvm}, the martingale posterior is strongly consistent at $\theta^*$, that is $\Pi(\theta_{n\infty} \in U^c \mid Y_{1:n}) \to 0$ as $n \to \infty$ $P^*$-almost surely for every neighborhood $U$ of $\theta^*$, where $\theta_{n\infty}$ is distributed according to the martingale posterior centred at $\theta_n$. 
\end{proposition}
\begin{proof}
  Let $\epsilon > 0$ be the size of the neighborhood $U$. We apply Markov's inequality:
    \begin{align*}
        \Pi\left(|\theta_{n\infty} - \theta^*|> \epsilon \mid Y_{1:n}\right) \leq \frac{1}{\epsilon^2}E\left\{(\theta_{n\infty} - \theta^*)^2 \mid Y_{1:n}\right\}.
    \end{align*}
    The expectation can be written as
    \begin{align*}
       E\left\{(\theta_{n\infty} - \theta^*)^2 \mid Y_{1:n}\right\}= E\left\{(\theta_{n\infty} - \theta_n)^2 \mid Y_{1:n}\right\}+  \left(\theta_n - \theta^*\right)^2,
    \end{align*}
    where the cross-term is zero from the martingale, and the first term is simply the posterior variance. We show in Lemma \ref{app:lem:sn_UI} of the Appendix that the posterior variance goes to 0 as $n \to \infty$. Finally, $\theta_n \to \theta^*$ $P^*$-almost surely, giving us the result.
\end{proof}

We see that posterior consistency follows from the posterior variance approaching 0 and consistency of the posterior mean. The above can be weakened to convergence in $P^*$-probability if $\theta_n$ is only weakly consistent. Finally, \cite{Fong2024} extends the above to derive posterior contraction rates in the nonparametric case, but we omit this in favour of brevity, and focus directly on the Bernstein-von Mises result.

\subsection{Martingale central limit theorem}

 We now consider the full Bernstein-von Mises result.
 To begin, we will consider a {deterministic} sequence $\left(\theta_n\right)_{n \geq 1}$ which converges to $\theta^*$, before considering the case where $\theta_n$ is estimated from $Y_{1:n} \iid P^*$.  We then have the following main result.
\begin{theorem}\label{th:bvm}
    Let $(\theta_n)_{n \geq 1}$ be a  deterministic sequence which converges to $\theta^*$ as $n \to \infty$. Under Assumptions \ref{as:existence} and \ref{as:bvm}, we have that
    \begin{align*}
   r_n^{-1}\left(\theta_{n\infty} - \theta_n\right) \to \mathcal{N}\left\{0,\mathcal{I}(\theta^*)^{-1}\right\}
    \end{align*}
    in distribution as $n \to \infty$, where $r_n^2 = {\sum_{i = n+1}^\infty i^{-2}}$.
    \end{theorem}
    \begin{proof}
    The key is to construct an infinite martingale array, where each row is the sequence $(\theta_{ni})_{i \geq 1}$ arising from predictive resampling starting at $\theta_{n0} = \theta_n$. In this case, we have one imputed population per $\theta_n$. Assumption \ref{as:bvm} then implies the required Lindeberg condition and variance convergence result for the martingale central limit theorem for infinite arrays (e.g. \citet[Theorem 3.6]{Hall2014}). We detail the full proof in Section \ref{sec:app_bvm} of the Appendix. 
    \end{proof}

We now extend the above to where $\theta_n$ is estimated from data and is no longer deterministic. 
In the below, we require strong consistency of the estimator due to the dependence on the deterministic sequence when applying the martingale central limit theorem; we leave the extension to weakly consistent estimators for future work as we believe it requires a novel proof technique.
\begin{theorem}\label{th:bvm_as}
    Suppose $\left(\theta_n\right)_{n \geq 1}$ is estimated from $Y_{1:n}$ drawn from $P^*$ and is strongly consistent at $\theta^*$. 
    Under Assumptions \ref{as:existence} and \ref{as:bvm}, we have
    \begin{align*}
       r_n^{-1}\left(\theta_{n\infty} - \theta_n\right) \to \mathcal{N}\left\{0,\mathcal{I}(\theta^*)^{-1}\right\} 
    \end{align*}
    in distribution $P^*$-almost surely as $n \to \infty$, where $r_n^2 = {\sum_{i = n+1}^\infty i^{-2}}$.
\end{theorem}
\begin{proof}
By assumption, we have that $\theta_n \to \theta^*$ $P^*$-almost surely. Consider a fixed $\omega \in A$ where $P^*(A) = 1$, then $\left(\theta_n(\omega):n \geq 1\right)$ is a deterministic sequence which converges to $\theta^*$, so we can apply Theorem \ref{th:bvm}. Since Theorem \ref{th:bvm} holds for any convergent sequence $\theta_n \to \theta^*$, the martingale central limit theorem holds with $P^*$-probability 1.
\end{proof}

\subsection{Centring and asymptotic variance}
For traditional Bayesian models, Bernstein-von Mises results are centred around an efficient sequence of estimators; that is, any sequence $\hat{\theta}_{n}$ satisfying
\begin{equation*}
    {n}^{1/2}(\hat{\theta}_{n} - \theta^{*}) = {{n}^{-1/2}}\sum_{i=1}^{n}\mathcal{I}(\theta^*)^{-1}s(\theta^{*}, Y_{i}) + o_{P^{*}}(1).
\end{equation*}
For such a sequence, we have ${n}^{1/2}(\hat{\theta}_{n} - \theta^{*}) \rightarrow \mathcal{N}\left\{0,\mathcal{I}(\theta^*)^{-1}\right\}$ in distribution, and since the asymptotic variance $\mathcal{I}(\theta^*)^{-1}$ matches that of the posterior, we deduce that central credible sets are also asymptotic confidence sets \citep[e.g.][]{van2000}. The efficient centring sequence arises automatically from the shape of the likelihood rather than the choice of prior. In contrast, our results above are always centred around the chosen initial estimate $\theta_{n}$ computed from $Y_{1:n}$. As a result, any guarantees with respect to $\theta^*$ instead depend solely on properties of the estimate $\theta_n$. In this section, we address the implications of this different form of centring and  assess the asymptotic variance for coverage.

  We continue to  assume the model is well-specified. The normalizing rate $r_n$ can  be replaced with ${n}^{-1/2}$, as they are asymptotically equivalent. Theorem \ref{th:bvm_as} then states that  $\pi(\theta_{n\infty}\mid Y_{1:n}) \approx \mathcal{N}\{\theta_n, n^{-1}\mathcal{I}(\theta^*)^{-1}\}$. We thus only have correct coverage of posterior credible intervals, as with the usual Bernstein-von Mises result, if $r_n^{-1} (\theta_n - \theta^*) \to \mathcal{N}\left\{0,\mathcal{I}(\theta^*)^{-1}\right\}$ in distribution as $n \to \infty$. Fortunately, asymptotic normality of maximum likelihood estimates \citep{van2000} or estimates obtained from stochastic approximation is well-studied \citep{Lai2003}, so we can show that $\theta_n$ obtained from (\ref{eq:sgd}) will satisfy the required  central limit theorem above if the model is well-specified and under additional regularity conditions.  As these are standard results, we leave this discussion for Section \ref{sec:init_estimate} in the Appendix.

 Additionally, our result highlights that preconditioning by $\mathcal{I}(\theta)^{-1}$ in (\ref{eq:sgd}) is responsible for 
  the magnitude of the asymptotic posterior variance, which is crucial for good frequentist properties of the induced parametric martingale posterior if $\theta_n$ is an efficient estimator.
  Meanwhile, the learning rate $N^{-1}$ in (\ref{eq:sgd}) dictates that the posterior variance shrinks at rate $n^{-1}$, so our recommended choice is again appropriate in this setting.

In summary, for parametric martingale posteriors, there is a noticeable divide between posterior centring and posterior uncertainty, as they are studied separately and can even be specified from distinct update rules/estimators.
This is noticeably different to traditional Bayesian models, where the point estimate is implied by the full posterior distribution. 
Reassuringly, under correct model specification and  regular settings, the frequentist properties of traditional Bayesian posteriors and parametric martingale posteriors indeed do match.

\section{Extensions}
\subsection{Multivariate parameters}\label{sec:mv}
Although the results so far are only for the univariate case, direct extensions to the multivariate case are just applications of the Cram\'{e}r-Wold device, although stable convergence will now play a greater role. The key difference is that the update term $Z_N$ as in (\ref{eq:Z}) will now be a vector which depends on the score vector and the Fisher information matrix, and we will work with the norms $\|Z(\theta,Y)\|^4$ and  $\|\theta\|^4$  for the analogous version of Proposition \ref{prop:Z4}. 
We provide the main results in Section \ref{sec:app_mv} of the Appendix, with examples in Section \ref{sec:emp}. 

\subsection{Beyond independent and identically distributed}
It is possible to extend the parametric martingale posterior to certain dependent settings, such as regression or time series models \citep{Walker2022,Holmes2023, Moya2024}. In this section, we outline extensions of the theory and methods to the regression setting. Let $X \in \mathcal{X} \subseteq \R^p$, and $Y \in \mathcal{Y}$ where $\mathcal{Y}\subseteq \R$ or $\mathcal{Y}= \{0,1\}$ in the continuous and binary case respectively, and let $p_{\theta}(y \mid x)$ denote the conditional model. We assume that $Y_i \mid X_i \sim p_{\theta^*}(y \mid x)$, and $(X_i)_{i \geq 1}$ is a deterministic covariate sequence (i.e. the fixed design setting).

As discussed in \cite{Holmes2023,Fong2023,Fortini2024}, it is most natural to devote modelling efforts to the conditional distribution of $Y \mid X$. For predictive resampling, we then draw $X_{n+1:\infty}$ from the deterministic covariate sequence, or independently and identically distributed from the empirical distribution of the observed covariates $X_{1:n}$ if the full sequence of covariates is not known.
In the random design case, we can use the Bayesian bootstrap to draw $X_{n+1:\infty}$, but we do not consider that here.

 Predictive resampling then involves first drawing $X_{N}$ from the chosen scheme,  then subsequently drawing $Y_{N}  \sim p_{\theta_{N-1}}(\cdot \mid X_{N})$, and finally computing 
\begin{align*}
    \theta_{N} = \theta_{N-1} + N^{-1}\mathcal{I}(\theta_{N-1})^{-1} s(\theta_{N-1}, Y_{N}; X_{N})
\end{align*}
where $s(\theta, Y; X) = \nabla_{\theta} \log p_{\theta}(Y \mid X)$ and 
\begin{align*}
    \mathcal{I}(\theta; X) = E_\theta\left\{s(\theta, Y; X)\, s(\theta, Y; X)^T \mid X\right\}, \quad  \mathcal{I}(\theta) = \lim_{n \to \infty}\frac{1}{n}\sum_{i = 1}^n\mathcal{I}(\theta; X_i).
    \end{align*}
One can verify quite easily that the choice of the method to impute $X_{n+1:\infty}$ does not impede the martingale condition.  In practice, if $\mathcal{I}(\theta)$ is not known, we can approximate $\mathcal{I}(\theta)$ with
    \begin{align*}
       \widehat{\mathcal{I}}_n(\theta) = \frac{1}{n}\sum_{i = 1}^n\mathcal{I}(\theta; X_i).
    \end{align*}
    In this case, $\widehat{\mathcal{I}}_n(\theta)$ is estimated once from the observed $X_{1:n}$ and held fixed during predictive resampling.
To obtain the central limit theorems, we can simply apply the multivariate extension of Theorems \ref{th:pred_CLT} and \ref{th:bvm_as}, where the martingale is with respect to the filtration generated by the pairs $(Y_N,X_N)_{N > n}$. We provide detailed results in Section \ref{sec:app_reg} of the Appendix, with an additional logistic regression example in Section \ref{sec:app_logreg} of the Appendix.

  \subsection{Model misspecification}\label{sec:misspec}
Unlike in traditional Bayes, Theorem \ref{th:bvm_as} holds true irrespective of whether the model is well-specified or not, as long as $\theta_n$ converges $P^*$-almost surely. This is another concrete demonstration of the separation between the centring and asymptotic variance of the parametric martingale posterior. We also highlight that the asymptotic variance of the parametric martingale posterior does not depend on $P^*$, so the uncertainty is more parametric in a sense.  In contrast, for traditional Bayes, under regularity conditions, the asymptotic variance of the regular Bayesian posterior is $J(\theta^*) = -E_{P^*}\{\partial s(\theta^*,Y)/ \partial \theta\}$ instead of $\mathcal{I}(\theta^*)$ \citep{Kleijn2012}. 

Under model misspecification, the asymptotic variance of the initial estimate $\theta_n$ will no longer match that of the parametric martingale posterior, so we will not attain correct coverage, which also occurs in the Bayesian case. However, the additional flexibility of the learning rate in (\ref{eq:sgd}) offers a potential solution, which is akin to a tempered likelihood in the regular Bayesian case \citep{Holmes2017}. Specifically, if we replace $N^{-1}$ in (\ref{eq:sgd}) with $a N^{-1}$ for $a \in \R^+$, we can appropriately inflate or shrink $\mathcal{I}(\theta^*)$ to match the asymptotic variance of the initial estimate $\theta_n$.  As the condition (\ref{eq:alpha}) will still be satisfied, all of our results will still hold.
 This can also be extended to the multivariate case, where $a$ is a chosen positive definite matrix instead of a scalar. Unlike tempering in regular Bayes, this gives us the freedom to adjust the entire covariance structure of the posterior distribution, instead of just a uniform re-scaling (e.g. as in \cite{Lyddon2019}).

\section{Empirical studies}\label{sec:emp}
\subsection{Setup}
We outline the implementation details for the empirical studies. We implement the parametric martingale posterior in the Julia programming language \citep{Bezanson2017}, which offers high efficiency for iterative updates. For baseline comparators requiring Markov chain Monte Carlo, we rely on the efficient No-U-Turn Sampler \citep{Hoffman2014} implemented in the \texttt{Turing.jl} package \citep{Ge2018}  in Julia. All experiments are run on an Apple M2 Pro CPU.

\subsection{Simulations}\label{sec:sim}

This section presents a simulation study which compares the parametric martingale posterior to traditional Bayesian inference in a bivariate normal model. For $y \in \R^2$, our predictive model is $p_\theta(y) = \mathcal{N}(y \mid \mu,\Sigma)$,
where  $\mu = [\mu_1, \mu_2]^T$ is a bivariate vector and $\Sigma$ is a $2 \times 2$ covariance matrix with entries $\Sigma_{11} = s_1$, $\Sigma_{22} = s_2$ and $\Sigma_{12} = \Sigma_{21} = s_{12}$. We thus consider $\theta = (\mu_1,\mu_2,s_1,s_2,s_{12})$. For the parametric martingale posterior, it is not too challenging to verify that the natural gradient takes the form $Z_{\mu_j}(\theta,Y) = Y_j - \mu_j$ for the means, $Z_{s_j}(\theta,Y) =(Y_j -\mu_j)^2- s_j$ for the variances and $Z_{s_{12}}(\theta,Y) = (Y_1- \mu_1)(Y_2 - \mu_2) - s_{12}$ for the covariance.
For hybrid predictive resampling, $\mathcal{I}(\theta)^{-1}$ can be easily computed, and the conditions required for the multivariate versions of Theorems \ref{th:pred_CLT} and \ref{th:bvm_as} can be verified.
We omit these details for brevity, with full derivations provided in Section \ref{sec:normal_mv} of the Appendix. For the initial estimate $\theta_n$, we use the sample estimates of $\mu$ and $\Sigma$ as they are unbiased and easy to implement, and negligibly different to (\ref{eq:sgd}). 

For the regular Bayesian model, we opt for weakly informative priors for all parameters \citep{Gelman2006,Gelman2008}. In particular, we have
\begin{align*}
    \mu_j \sim \mathcal{N}(0,10^2), \quad {s_j}^{1/2} \sim \text{Half-Cauchy}(0,5^2), \quad \rho \sim \text{Unif}(-1,1)
\end{align*}
where $s_{12} = \rho ({s_1 s_2})^{1/2}$. We highlight that it is not straightforward to elicit prior distributions on covariance matrices which are noninformative or weakly informative \citep{Fuglstad2020,Pinkney2024}. On the other hand, this is not a necessary consideration for the parametric martingale posterior, which is entirely prior-free. Furthermore, due to the martingale, the posterior mean of the martingale posterior will always be equal to the initial estimate $\theta_n$.  

We set the truth to $\theta^* = (-0.5,1.0,1.0,0.5,0.7)$, and evaluate the coverage and average length of 95\% credible intervals for different sample sizes $n = (20,50,500)$. For each  setting of $n$, we consider 5000 sampled datasets for the parametric martingale posterior. Due to the requirement of Markov chain Monte Carlo and its associated cost, we only consider 500 sampled datasets for the Bayesian posterior. We generate $B = 2000$ posterior samples for both methods, and use a truncation of $N = n+50$ for hybrid predictive resampling.

Table \ref{tab2} illustrates the results of the simulations for $(\mu_1,s_1,s_{12})$, where we leave $(\mu_2, s_2)$ for the Appendix as the results are practically the same.
We see that the parametric martingale posterior undercovers slightly for $n = 20$, but attains accurate coverage once $n$ is increased, in agreement with Theorem \ref{th:bvm_as}. 
The Bayesian posterior on the other hand attains coverage for all $n$ at the cost of noticeably wider intervals. This discrepancy between the two posteriors disappears as $n$ is increased, which is unsurprising as they are asymptotically equivalent. This agrees with the observation made by \cite{Moya2024} that the Bayesian and martingale posterior variance differ by order $n^{-2}$, which is likely the effect of the prior distribution.
The above results highlight the value of a well-specified prior when $n$ is small, but it is obvious that a poorly specified prior distribution would greatly hurt coverage. The parametric martingale posterior may be more suitable when $n$ is of reasonable size or there is little prior information, and it is in a sense a default or `frequentist' posterior \citep{Holmes2023}.

\begin{table}
    \caption{Coverage and average length of 95\% credible intervals of traditional Bayesian posterior (Markov Chain Monte Carlo) and the parametric martingale posterior (hybrid predictive resampling) on different sample sizes 
}
\center
\begin{small}
{   \begin{tabular}{ccrrcrrcrr}
        && \multicolumn{2}{c}{$n=20$} & &\multicolumn{2}{c}{$n=100$}&&\multicolumn{2}{c}{$n=500$}\vspace{2mm} \\
        Method& Parameter & Cov & Length&&Cov & Length && Cov &  \\ 
        \multirow{5}{*}{Bayesian posterior}&$\mu_1$ & 93.6& 9.2&& 92.6 & 3.9&& 94.2&1.7\\ 
& $s_1$& 94.8 & 15.3 && 95.8 &5.8&& 93.0 & 2.5\\
& $s_{12}$& 95.2 & 10.8&& 95.4& 4.1 && 93.6 & 1.8\\[5pt]
   &{Max. SE} & 1.1& 0.2 && 1.1 & 0.04 && 1.1 &0.008\\
   & Run-time &   \multicolumn{2}{c}{1.2s}&&\multicolumn{2}{c}{3.7s}&& \multicolumn{2}{c}{14.7s} \\[7pt]
        \multirow{5}{*}{Martingale posterior} 
  & $\mu_1$ & 93.0 & 8.6&& 93.3 & 3.9&& 94.4& 1.7\\ 
    & $s_1$ &  91.3 & 12.0&& 94.4& 5.5&& 94.8& 2.5\\ 
      & $s_{12}$ &90.0 &8.1 && 94.1& 3.8&& 94.6& 1.7 \\[5pt]
      &{Max. SE} & 0.4& 0.06  && 0.4 &0.01 &&   0.3 & 0.002\\[2pt]
         & Run-time &   \multicolumn{2}{c}{0.003s} && \multicolumn{2}{c}{0.003s}&& \multicolumn{2}{c}{0.003s}\vspace{5mm}\\
         \multicolumn{10}{l}{\footnotesize Cov, coverage;  SE, standard error. Coverage is in \%; length has been multiplied by 10. }\\
         \multicolumn{10}{l}{\footnotesize Run-time is per sample; results are over 500 and 5000 repeats for Bayes and the martingale}\\
         \multicolumn{10}{l}{\footnotesize posterior respectively.}
    \end{tabular}}
    \end{small}

    \label{tab2}
\end{table}

A much more stark contrast appears for computation time, where the parametric martingale posterior requires orders of magnitude less time than traditional Bayes. This highlights the massive computational gains made possible by side-stepping Markov chain Monte Carlo. Another benefit of the martingale posterior sampling scheme is that the samples are independent, unlike the output from Markov chain Monte Carlo. Finally, the run-time for the  martingale posterior does not appear to increase much with $n$. In fact, we could technically decrease $N$ as $n$ increases, so the time for hybrid predictive resampling would actually decrease with $n$. The computation time of the martingale posterior is thus bottle-necked by the time for the initial estimation $\theta_n$, which is still negligible at the above sample sizes, but may become noticeable for sufficiently large $n$.

    \subsection{Real data example}\label{sec:rwd}
We demonstrate the parametric martingale posterior on a  robust regression task, using the AIDS Clinical Trials Group Study 175
dataset from the UCI Machine Learning Depository \citep{Asuncion2007,Hammer1996}. The dataset consists of $n = 2139$ patients with HIV who were randomized to 4 treatment arms: (1) zidovudine + didanosine; (2) zidovudine + zalcitabine; (3) didanosine; (4) zidovudine. Following a similar setup to \cite{Hines2022,Li2023},  we consider the continuous outcome as the CD4 count at $20\pm5$ weeks, which is a measure of immune function after some time on treatment. In addition to the treatment group variable, we consider 12 baseline variables given in Section \ref{sec:app_rwd} of the Appendix.
For pre-processing, we normalize all continuous covariates and the outcome to have mean 1 and standard deviation 1, and  introduce 3 dummy variables for the treatment group variable, with zidovudine as the reference. Including the intercept, our design matrix thus consists of $p = 16$ variables. 

A quick check reveals that the excess kurtosis of CD4 count at baseline is equal to 1.8, indicating that the outcome may also be heavy-tailed. We thus opt to use the heavy-tailed linear Student-$t$ model \citep{Lange1989,Geweke1993}. Let $\theta =[\beta, \tau^2]^T$ where $\beta\in \mathbb{R}^{p}$ are the regression coefficients (including an intercept) and $\tau \in \mathbb{R}$ is the scale parameter. Suppose we observe $(Y_{1:n}, X_{1:n})$ where $Y_i \in \mathbb{R}, X_i \in \mathbb{R}^p$. The model is then 
\begin{align*}
   Y_i = \beta^T X_i + \tau \varepsilon_i,  \quad \varepsilon_i \sim \text{Student-}t(\nu),
\end{align*}
where
$\nu > 1$ is the fixed degrees of freedom, which gives $p_{\theta}(y \mid x)$ as the location-scale Student-$t$ distribution. A natural plug-in predictive $p_{\theta_n}(y \mid x)$ where $\theta_n$ is the maximum likelihood estimate can then be used for the parametric martingale posterior.
We highlight that in this case, the estimation of $\theta_n$ requires slightly more work due to the non-concavity of the log-likelihood: we utilize iteratively reweighted least squares \citep{Lange1989} repeated with multiple initializations to estimate the global maximum. 
The asymptotic theory then follows from standard $M$-estimation theory.
Finally, we treat $\nu$ as a hyperparameter and set it to $\nu = 5$. In practice, we could also select it by maximizing the prequential log-likelihood  \citep{Dawid1984}
taking the form $\sum_{i = 1}^n \log p_{\theta_{i-1}}(Y_i \mid X_i)$,
which is closely connected to the marginal likelihood \citep{Gneiting2007,Fong2020}. 

\begin{figure}
\begin{center}
\includegraphics[width=0.55\textwidth]{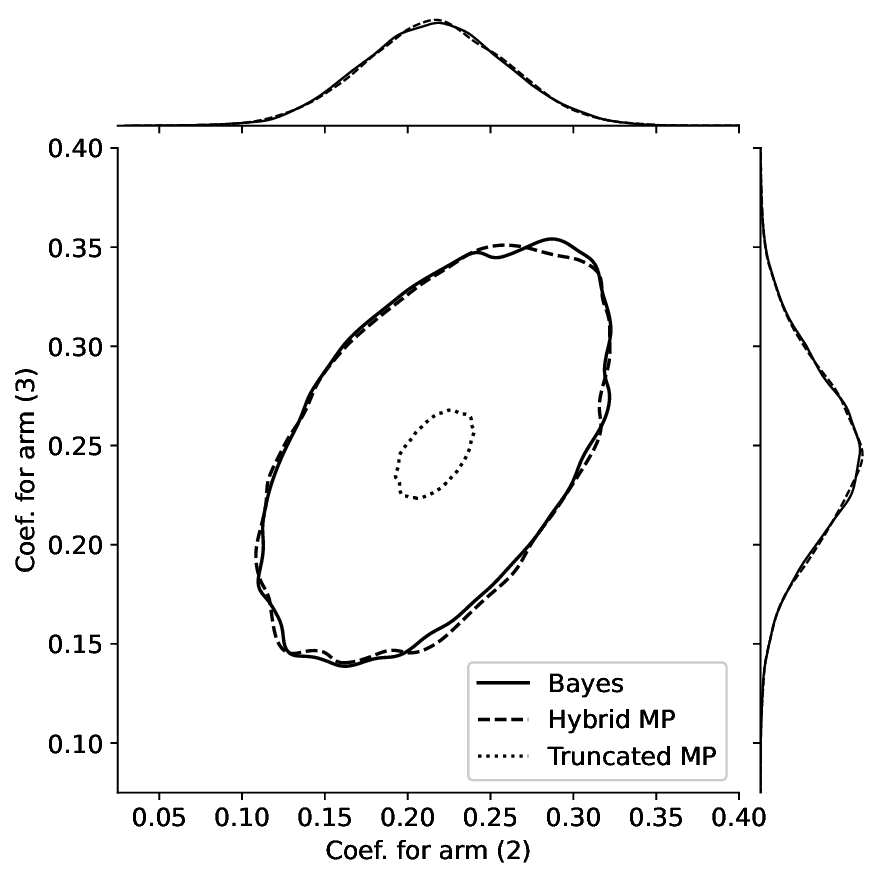}
\end{center}
\vspace{-7mm}
\caption{Posterior 95\% probability contour of the coefficients for treatment arms (2) and (3)
for the Bayes posterior (solid), MP with hybrid sampling (dashes), and MP with truncation $N = n+100$ (dot); 
Marginal kernel density plots are shown above/right, where we omit the truncated MP for clarity;
MP, martingale posterior.}
\label{fig:biv_aids}
\end{figure}

For the parametric martingale posterior, we keep ${\nu}=5$ fixed and set the initial estimate  $\theta_n$  as the maximum likelihood estimate obtained from 10 repeated optimizations. We truncate at $N = n+ 50000$ and $N = n+100$ for exact and hybrid predictive resampling respectively.  One can show that the natural gradient is $Z_n(\theta, Y ;X) = [Z_{\beta}(\theta,Y;X)^T, \,\, Z_{\tau^2}(\theta,Y;X)]^T$ where
\begin{align*}
    Z_{\beta}(\theta,Y;X) &= \left\{ \frac{\tau(\nu + 3)R}{ \nu  +R^2}\right\}\Sigma_{n,x}^{-1}X,
\quad Z_{\tau^2}(\theta,Y;X) =  \frac{\tau^2(\nu + 3)\left(R^2 - 1\right)}{\nu  +R^2}
\end{align*}
for $R = (Y- \beta^T X)/\tau$, and $\Sigma_{n,x} = n^{-1}\sum_{i = 1}^n X_i X_i^T$. We show in Section \ref{sec:app_robreg} of the Appendix that the above update satisfies all conditions for the regression extension of Theorem \ref{th:pred_CLT} as long as $\Sigma_{n,x}$ is positive definite. Under weak assumptions on the sequence of design points $X_{1:n}$ {and additional regularity conditions for consistency of the maximum likelihood estimate}, we also have the regression extension of Theorem \ref{th:bvm_as}. For the Bayesian model, we set the same value of ${\nu}=5$, and again elicit the weakly informative priors 
$$\beta_j \sim \mathcal{N}(0,10^2), \quad \tau \sim \text{Half-Cauchy}(0,5^2).$$ We generate $B = 10000$ posterior samples in both cases.

Figure \ref{fig:biv_aids} illustrates the posterior 95\% contours for $(\beta_j,\beta_k)$ corresponding to  treatment arms (2) and (3) under the Bayes posterior, the martingale posterior with truncation $N =  n+100$ and the hybrid scheme. We exclude the plot of the exact scheme ($N = n+50000$) as this is visually indistinguishable from the hybrid scheme - we provide this comparison in the Appendix.
We see that the martingale posterior truncated at $N = n+100$ drastically underestimates the uncertainty, but this is corrected by the addition of the Gaussian correction in the hybrid scheme. 
The martingale posterior with hybrid predictive resampling is very similar to to the Bayes posterior, which is unsurprising given the weakly informative prior. We see that both the Bayesian and martingale posterior suggests that treatment arms (2) and (3) are more effective than the reference arm (4), which is in agreement with \cite{Hammer1996}.

Table \ref{tab3} illustrates the run-times. Regular Bayes with Markov chain Monte Carlo required 132.4s for posterior sampling. On the other hand, the martingale posterior with $N = n + 50000$ and $N = n + 100$ required 14.4s and 0.03s respectively, with  hybrid predictive resampling requiring a negligible extra amount of time compared to the latter. As before, the run-time for predictive resampling is roughly proportional to $N$, and the hybrid case is the fastest by a very large margin. 
\begin{table}[!ht]
    \caption{Run-times on the ACTG 175 dataset }
    \center
    \small
{   \begin{tabular}{cc}
        Method& \hspace{4mm}Run-time \hspace{4mm} \vspace{1mm}\\ 
        Bayes &132.4s\\
            MP ($N = n+50000$)& 14.4s\\
            MP (Hybrid, $N = n+100$) & 0.03s \vspace{2mm}\\
            \multicolumn{2}{l}{\footnotesize MP, martingale posterior. }
    \end{tabular}}
    \label{tab3}
\end{table}

\section{Discussion}
In this work, we have leveraged martingale central limit theorems to accelerate predictive resampling and provide frequentist guarantees for the parametric martingale posterior. The experiments demonstrate that the parametric martingale posterior can be significantly faster than traditional Bayes, with results that match that of a posterior arising from a noninformative prior. 
We now outline a few interesting future directions of research. Firstly, it would be interesting to consider how one might use the parametric martingale posterior when there is informative prior information, which could open up predictive resampling as a posterior approximation scheme. It is also of much interest to extend the above methodology to models with structure such as hierarchical models, where Bayesian methods arguably have the most application. Finally, it will be interesting to see if the above results and hybrid predictive resampling scheme can be applied to the general nonparametric setting beyond that considered in \cite{Fong2024}.\vspace{5mm}

\section*{Acknowledgement}
A. Yiu receives funding from Novo Nordisk. We thank Chris Holmes and Ziyu Wang for the interesting discussions.

\section*{Code}
Code for reproducing the results can be found at \url{https://github.com/edfong/parametric_MP}.

\bibliographystyle{apalike}
\bibliography{paper-ref}

\newpage

\begin{appendices}

\setcounter{equation}{0}
\setcounter{theorem}{0}
\setcounter{proposition}{0}
\setcounter{lemma}{0}
\setcounter{corollary}{0}
\setcounter{assumption}{0}

\renewcommand{\theequation}{\thesection.\arabic{equation}}
\renewcommand{\thetheorem}{\thesection.\arabic{theorem}}
\renewcommand{\theproposition}{\thesection.\arabic{proposition}}
\renewcommand{\thelemma}{\thesection.\arabic{lemma}}
\renewcommand{\thecorollary}{\thesection.\arabic{corollary}}
\renewcommand{\theassumption}{\thesection.\arabic{assumption}}

\renewcommand{\thetable}{\thesection.\arabic{table}}
\renewcommand{\thefigure}{\thesection.\arabic{figure}}

\section{Prerequisites}
\subsection{Stable convergence}\label{sec:stab}
In this section, we briefly outline the concept of stable convergence, and provide some useful technical results for some of the later proofs, in particular for the multivariate result. We recommend \cite{Hausler2015} for a detailed exposition of stable convergence and limit theorems, and the discussion provided in \cite{Crimaldi2009, Fortini2020,Fortini2024} for the use of stable convergence within predictive asymptotics. 

Stable convergence, as introduced in \citet{Renyi1963}, is a notion of convergence of random variables that is stronger than convergence in distribution and weaker than convergence in probability. It plays a particularly strong role in martingale central limit theorems, where normings in the central limit theorem can often have random limits.
For studying the predictive asymptotics of $(\theta_\infty - \theta_N)$ through predictive resampling, the limiting distribution may have random components (e.g. the variance) as the parameter that we are converging to is random. This is in stark contrast to classical asymptotics, where $\theta$ is converging to some fixed $\theta^*$, and normings usually have deterministic limits.

Let $(X_N)_{N \geq 1}$ and $({\eta}_N)_{N \geq 1}$ denote sequences of random variables. Suppose there exists another random variable $\eta$, which is positive almost surely, such that ${\eta}_N \to \eta$ in probability.  Intuitively, stable convergence is exactly what allows us to replace a statement like
$X_N \to \mathcal{N}(0,\eta^2)$ with ${\eta}_N^{-1}X_N \to \mathcal{N}(0,1)$. In the case where $\eta$ is a constant and $X_N$ converges in distribution, the above is just an application of Slutsky's theorem, where ${\eta}_N$ is a weakly consistent estimator of the unknown $\eta$. However, when $\eta$ is random, convergence of distribution is insufficiently strong for this interchanging of normings, with stable convergence being the appropriate strengthening; see \citet[Chapter 1]{Hausler2015} for some examples on this. The interchanging of normings is key to writing limit theorems in a form that is useful for statistical inference, or in our case for approximating posterior sampling.

We now provide a few definitions with sufficient generality for our needs. Let $(\Omega,\mathcal{F}, P)$ be a probability space, and $\mathcal{G}\subseteq \mathcal{F}$ be a sub-$\sigma$-algebra. Furthermore, let $(X_N)_{N \geq 1}$ and $X$ denote $(\mathbb{R}, \mathcal{B}(\mathbb{R}))$-valued random variables on the same probability space. We say that $X_N \to X$ $\mathcal{G}$-stably if for every $f \in L_1(P)$ and $h \in C_B(\mathbb{R})$, we have
\begin{align*}
  E\left[fE\left\{h(X_n) \mid \mathcal{G}\right\}\right] \to  E\left[fE\left\{h(X) \mid \mathcal{G}\right\}\right].
\end{align*}
If in addition to stability, we have that $\sigma(X)$ and $\mathcal{G}$ are independent, then we say that $X_N \to X$ $\mathcal{G}$-mixing. There are many different equivalent definitions of stable convergence, e.g. \citet[Theorem 3.17]{Hausler2015}. In terms of conditional probability distributions, the following condition may be easier to interpret. For each $F \in  \mathcal{G}$ with $P(F) > 0$,  assume that for each continuity point $x$ of $X$, we have
\begin{align*}
P(X_n\leq x \mid F) \to P(X \leq x \mid F)
\end{align*}
where $P(\cdot \mid F) = P(\cdot \cap F )/P(F)$ is the conditional probability given $F$.
It is clear that this is a strengthening of convergence in distribution, as choosing $\mathcal{G} = \{\emptyset, \Omega\}$ returns us the regular convergence in distribution as $P(\cdot \mid{\Omega}) = P(\cdot)$. Stable convergence then strengthens convergence in distribution by choosing a larger sub-$\sigma$-algebra $\mathcal{G}$.
Fortunately, the convergence in martingale central limit theorems is often already shown to be stable, so we will not need to check this directly.

More important to us are the consequences of stable convergence. In particular, we will need the following three results. The first is the continuous mapping result for stable convergence.
The second is exactly the analogy of Slutsky's theorem when the norming is random, and the third is a version of the Cram\'{e}r-Wold theorem, which we will need for the multivariate case later on. We present specialized versions of the results for convenience, with the general case in reference.

\begin{proposition}[{{\citet[Theorem 3.18]{Hausler2015}}}]\label{app:prop:cont_map_stable}
    Let $X_n$ and $X$ be $(\mathbb{R}^d, \mathcal{B}(\mathbb{R}^d))$-valued random variables, where $X_n \to X$ $\mathcal{G}$-stably. If $g: \mathbb{R}^d \to \mathbb{R}^p$ is a measurable  $P_X$-almost surely continuous function, then 
\begin{align*}
    g(X_n) \to g(X)
\end{align*}
$\mathcal{G}$-stably.
\end{proposition}

\begin{proposition}[{{\citet[Theorem 3.18, Corollary 6.3]{Hausler2015}}}]\label{app:prop:slutsky_stable}
Let $X_n$ and $X$ be $(\mathbb{R}^d, \mathcal{B}(\mathbb{R}^d))$-valued random variables, where $X_n \to X$ $\mathcal{G}$-stably. Furthermore, let $Y_n$ and $Y$ be $(\mathbb{R}, \mathcal{B}(\mathbb{R}))$-valued random variables, and suppose $Y_n \to Y$ in probability where $Y$ is $\mathcal{G}$-measurable. We then have
\begin{align*}
    (X_n,Y_n) \to (X,Y)
\end{align*}
$\mathcal{G}$-stably. Continuous mapping then implies
\begin{align*}
    X_n Y_n \to XY
\end{align*}
$\mathcal{G}$-stably. Furthermore, if $Y > 0$ almost surely, then we have
\begin{align*}
    \frac{X_n}{Y_n} \to \frac{X}{Y}
\end{align*}
$\mathcal{G}$-stably.
\end{proposition}

\begin{proposition}[{{\citet[Corollary 3.19]{Hausler2015}}}]\label{app:prop:cramwold_stable}
Let $X_n$ and $X$ be $(\mathbb{R}^d, \mathcal{B}(\mathbb{R}^d))$-valued random variables, where $X_n \to X$ $\mathcal{G}$-stably. Suppose we have
\begin{align*}
    t^T X_n \to  t^T X
\end{align*}
$\mathcal{G}$-stably for every $t \in \mathbb{R}^d$. Then $X_n \to X$ $\mathcal{G}$-stably.
\end{proposition}
In this work, we will only be concerned with the stable convergence of martingales. Let $\left(S_N, \mathcal{F}_N\right)_{N \geq 1}$ be a martingale bounded in $L_2$, and let $\mathcal{F}_\infty = \sigma\left(\cup_{N\geq 1}\mathcal{F}_N\right)$. In our work, we will let $\mathcal{G} = \mathcal{F}_\infty$, the random variable $X_n$ will be a scaled version of $S_N$, and $Y$ will be the sum of conditional variances of the martingale.  
An excellent detailed exposition of stable limit theorems for martingales is provided in \citet[Chapter 6]{Hausler2015}. 

\subsection{Tail sum martingales}\label{sec:app_tail}
In this subsection, we introduce a few standard results in connection with  martingale central limit theorems for tail sums. These results can be used to study the convergence of P\'{o}lya urns, so it is unsurprising that they will help us to study the predictive asymptotics of the parametric martingale posterior.

We first introduce some notation for tail sum martingales.
Let $\left(S_N, \mathcal{F}_N\right)_{N \geq 1}$ be a martingale bounded in $L_2$, and let $\mathcal{F}_\infty = \sigma\left(\cup_{N\geq 1}\mathcal{F}_N\right)$. Denote the martingale differences by $X_i = S_i - S_{i-1}$, where $X_0  = S_0$ is a constant, and define the conditional and unconditional variance respectively:
\begin{align}\label{app:eq_condit_var}
V_N^2 = \sum_{i=N+1}^\infty E\left(X_i^2 \mid \mathcal{F}_{i-1}\right), \quad s_N^2 = \sum_{i=N+1}^\infty E\left(X_i^2\right).
\end{align}
Both terms above are finite (almost surely for $V_N^2$) as the martingale is bounded in $L_2$. The martingale central limit theorem for tail sums has taken on quite a few different forms with various assumptions and choice of norming.
\cite{Heyde1977,Johnstone1978} and \citet[Corollary 3.5]{Hall2014} are some earlier examples. A detailed comparison of different forms of assumptions is given in \citet[Chapter 3.2]{Hall2014} and \citet[Chapter 6.3]{Hausler2015}.
For our purposes, the following version based on the conditional variance $V_N^2$ and a conditional Lindeberg assumption will be the most helpful.
\begin{theorem}[{{\citet[Exercise 6.7]{Hausler2015}}}]\label{app:th:tailsum_CLT}
Let $\left(S_N, \mathcal{F}_N\right)_{N \geq 1}$ be a martingale bounded in $L_2$ with martingale differences $X_i$, and let $\mathcal{F}_\infty = \sigma\left(\cup_{N\geq 1}\mathcal{F}_N\right)$. Furthermore, let $V_N^2$ and $s_N^2$ be defined as in (\ref{app:eq_condit_var}).
Assume that
\begin{align}\label{app:tailsum_as1}
        s_N^{-2}\sum_{i = N+1}^\infty E\left\{X_i^2 \mathbbm{1}\left(X_i^2> \varepsilon^2 s_N^2\right) \mid \mathcal{F}_{i-1}\right\} \to 0 
\end{align}
in probability for all $\varepsilon > 0$, and
\begin{align*}
    s_N^{-2}V_N^2 \to \eta^2 
\end{align*}
in probability for some random variable $\eta^2$ where $\eta^2 >0$ almost surely. Under the above assumptions, we have that
\begin{align}\label{app:tailsum_res1}
    s_N^{-1}\sum_{i = N+1}^\infty X_i \to \mathcal{N}(0,\eta^2)
\end{align}
$\mathcal{F}_\infty$-stably. 
\end{theorem}
In our use cases, the unconditional Lindeberg condition will be easier to check, which we can show implies the conditional Lindeberg condition as given below.
\begin{lemma}[Unconditional Lindeberg condition]\label{app:lem:uncondit_lind}
Consider following the unconditional Lindeberg condition. Assume that for all $\varepsilon > 0$,
    \begin{align} \label{app:tailsum_as3} 
s_N^{-2}\sum_{i = N+1}^\infty E\left\{X_i^2 \mathbbm{1}\left(X_i^2> \varepsilon^2 s_N^2\right)\right\} \to 0.
    \end{align}
    Then (\ref{app:tailsum_as1}) holds true.
\end{lemma}
\begin{proof}
The proof is based on \cite{Ganssler1978}, which we include for completion. See also \citet[Remark 6.8]{Hausler2015}. For each $\varepsilon > 0$, we have from     (\ref{app:tailsum_as3}):
    \begin{align*}
        s_N^{-2}\sum_{i = N+1}^{\infty} E\left\{X_i^2 \mathbbm{1}\left(X_i^2> \varepsilon^2 s_N^2\right)\right\} = E\left[s_N^{-2}\sum_{i = N+1}^{\infty} E\left\{X_i^2 \mathbbm{1}\left(X_i^2> \varepsilon^2 s_N^2\right)\mid \mathcal{F}_{i-1}\right\}\right]\to 0,
    \end{align*}
    where we have applied Tonelli's theorem. Since the inner term is non-negative, this implies 
    $$s_N^{-2}\sum_{i = N+1}^{\infty} E\left\{X_i^2 \mathbbm{1}\left(X_i^2> \varepsilon^2 s_N^2\right)\mid \mathcal{F}_{i-1}\right\}\to 0$$
    in probability.
\end{proof}

\subsection{Infinite martingale arrays}\label{sec:app_infmart}
For the Bernstein-von Mises result, we will require a different construction. In this case, we will have both $N \to \infty$ and $n \to \infty$, so an infinite martingale array is a natural choice. Martingale central limit theorems for infinite arrays is a straightforward generalization of the standard martingale central limit theorem, and is analogous to the classical Lindeberg-Feller central limit theorem, e.g. see \citet[Theorem 3.6]{Hall2014}. Stable convergence will not play as important of a role here, as we will generally have a limiting variance which is deterministic. 

We begin with some notation. For each $n \geq 1$, let $\left(S_{ni}, \mathcal{F}_{ni}\right)_{i \geq 1}$ denote a square-integrable martingale, so $\left(S_{ni}, \mathcal{F}_{ni}: i \geq 1, n \geq 1\right)$ is a square-integrable infinite martingale array.
We define $X_{ni} = S_{ni}- S_{n,i-1}$, where $X_{n0} = S_{n0}$ are constants, and  assume that
 \begin{align} \label{app:infin_array_as1} 
\sum_{j = 1}^\infty E\left(X_{nj}^2\right) < \infty
\end{align}
for each $n\geq 1$, so each row is a martingale bounded in $L_2$ as $s_n^2<\infty$.  This in turn implies that  for each $n$, there exists $S_{n\infty}$ such that  $S_{ni} \to S_{n\infty}$ almost surely and in $L_2$  as $i \to \infty$. Furthermore, the above guarantees the following are finite almost surely:
\begin{align}\label{app:eq_condit_var_n}
V_{n}^2 = \sum_{i=1}^\infty E\left(X_{ni}^2 \mid \mathcal{F}_{n,i-1}\right), \quad s_{n}^2 = \sum_{i= 1}^\infty E\left(X_{ni}^2\right).
\end{align}
 We now state the result martingale central limit theorem for one-sided infinite martingale arrays.
\begin{theorem}[{{\citet[Exercise 6.2]{Hausler2015}}}]\label{app:th:infinite_array_clt}
Let $\left(S_{ni}, \mathcal{F}_{ni}; i \geq 1, n \geq 1\right)$ be a square-integrable infinite martingale array, with martingale differences $X_{ni}$ which satisfy (\ref{app:infin_array_as1}). Furthermore, let $V_n^2$ and $s_n^2$ be defined according to (\ref{app:eq_condit_var_n}).
Assume that for all $\varepsilon > 0$,
 \begin{align*}
\sum_{i = 1}^{\infty} E\left\{X_{ni}^2 \mathbbm{1}\left(X_{ni}^2> \varepsilon^2 \right)\mid \mathcal{F}_{n,i-1}\right\} \to  0
\end{align*}
in probability, and
    \begin{align*} 
       V_{n}^2 \to \eta^2
\end{align*}
in probability, where $\eta^2 > 0$ is a positive constant almost surely. Under the above assumptions, we have that
\begin{align*}
    \sum_{i = 1}^{\infty} X_{ni}\to\mathcal{N}(0,\eta^2)
\end{align*}
in distribution.
\end{theorem}
\begin{proof}
    The above is shown in \citet[Exercise 6.2]{Hausler2015} for a general $\mathcal{G}_\infty$-measurable $\eta^2$, where $\mathcal{G}_\infty = \sigma\left(\cup_{n = 1}^\infty \cap_{m \geq n}\mathcal{F}_{m, k_n} \right)$. As $\eta^2$ is assumed to be constant, it is automatically $\mathcal{G}_\infty$-measurable \cite[Remark 6.2(b)]{Hausler2015}, so we can apply  \citet[Theorem 6.1]{Hausler2015} without the need to consider the technical nesting condition. We can then drop stability as the norming no longer has a random limit. 
\end{proof}

\setcounter{equation}{0}
\setcounter{theorem}{0}
\setcounter{proposition}{0}
\setcounter{lemma}{0}
\setcounter{corollary}{0}
\setcounter{assumption}{0}
\section{Proof of Proposition \ref{prop:Z4}}
We begin with a proof of Proposition \ref{prop:Z4}.
By assumption, we have $E\left(Z_N^4 \mid \mathcal{F}_{N-1}\right)  \leq  B+ C\theta_{N-1}^4 $ for all $N\geq 1$. H\"{o}lder's inequality gives
\begin{align*}
 E\left(|Z_N|^3 \mid \mathcal{F}_{N-1}\right)  \leq E\left(Z_N^4 \mid \mathcal{F}_{N-1}\right)^{{3}/{4}}, \quad 
    E\left(Z_N^2 \mid \mathcal{F}_{N-1}\right)\leq E\left(Z_N^4 \mid \mathcal{F}_{N-1}\right)^{{1}/{2}}.
\end{align*}
As we have $(a+b)^p \leq a^p + b^p$ for $0 < p < 1$ and $a,b > 0$, this gives
\begin{align*}
 E\left(|Z_N|^3 \mid \mathcal{F}_{N-1}\right)  \leq B^{{3}/{4}} + C^{{3}/{4}}\theta_{N-1}^3 , \quad 
    E\left(Z_N^2 \mid \mathcal{F}_{N-1}\right)\leq B^{{1}/{2}} +  C^{{1}/{2}}\theta_{N-1}^2 .
\end{align*}
We can then compute
\begin{align*}
    E\left(\theta_{N}^4\mid \mathcal{F}_{N-1}\right) &= E\left[ \left\{\theta_{N-1} + N^{-1} Z_{N}\right\}^4 \mid \mathcal{F}_{N-1}\right]\\
    &= \theta_{N-1}^4 + 6 N^{-2} \theta_{N-1}^2 E\left(Z_N^2 \mid \mathcal{F}_{N-1}\right) + 4 N^{-3}  \theta_{N-1} E\left(Z_N^3 \mid \mathcal{F}_{N-1}\right)\\
    &+ N^{-4} E\left(Z_N^4 \mid \mathcal{F}_{N-1}\right) 
\end{align*}
where we have used the fact that $E\left( Z_N \mid \mathcal{F}_{N-1}\right) = 0$. Plugging in the upper bounds,  the following holds for all $N \geq 1$:
\begin{align*}
     E\left(\theta_{N}^4\mid \mathcal{F}_{N-1}\right) &\leq \theta_{N-1}^4 \left(1 + 6 N^{-2} C^{{1}/{2}} +4 N^{-3} C^{{3}/{4}} + N^{-4} C\right)+ 6 N^{-2} B^{{1}/{2}}\theta^2_{N-1} \\
     &+ 4 N^{-3}B^{{3}/{4}} \theta_{N-1}+ N^{-4} B 
\end{align*}
One can show that both $\theta^2$ and $|\theta|$ are upper bounded by $\theta^4 + 1$, so we can write
\begin{align*}
     E\left(\theta_{N}^4\mid \mathcal{F}_{N-1}\right) &\leq \theta_{N-1}^4 \left\{1 + 6 N^{-2} \left(C^{{1}/{2}} + B^{{1}/{2}}\right) +4 N^{-3} \left(C^{{3}/{4}} + B^{{3}/{4}}\right) + N^{-4} C\right\}  + 6 N^{-2} B^{{1}/{2}} \\
     &+ 4N^{-3} B^{{3}/{4}} + N^{-4} B\\
     &\leq \theta_{N-1}^4 \left(1 + D_1 N^{-2}\right)  + D_2 N^{-2} 
\end{align*}
where there exists some non-negative $D_1,D_2 < \infty$ such that  the above holds for all $N \geq 1$. Iterated expectations give us
\begin{align*}
    E\left(\theta_N^4\right) &\leq \theta_n^4 \prod_{i = n+1}^N \left(1 + D_1i^{-2}\right) + D_2\sum_{i = n+1}^N i^{-2} \prod_{j = i+1}^N \left(1 + D_1j^{-2}\right).
\end{align*}
We have
\begin{align*}
    \prod_{i = n+1}^N \left(1 + D_1i^{-2}\right)  &= \exp\left\{\sum_{i = n+1}^N \log \left(1 + D_1i^{-2}\right)\right\} \leq  \exp\left(D_1\sum_{i = n+1}^N i^{-2}\right).
\end{align*}
As $\sum_{i = n+1}^\infty i^{-2} < \infty$, the above gives $\prod_{i = n+1}^\infty\left(1 + D_1i^{-2}\right) <\infty$, which in turn implies that there exists some non-negative $D_3,D_4 < \infty$ such that
\begin{align*}
     \sup_{N > n} E\left(\theta_N^4\right)t &\leq D_3\theta_n^4  + D_4 \sum_{i = n+1}^\infty i^{-2} < \infty.
\end{align*}
Finally, we have 
$$\sup_{N > n}E\left(Z_{N+1}^4\right) \leq  C\sup_{N > n}E\left(\theta_N^4\right) + B < \infty.$$
This is sufficient for $\left(Z_N^2\right)_{N > n}$ to be uniformly integrable.

\setcounter{equation}{0}
\setcounter{theorem}{0}
\setcounter{proposition}{0}
\setcounter{lemma}{0}
\setcounter{corollary}{0}
\setcounter{assumption}{0}
\section{Proof of Lemma \ref{lem:L2}}
   It is not hard to see that
 $E\left(Z_N^2 \mid \mathcal{F}_{N-1}\right) = \mathcal{I}(\theta_{N-1})^{-1}$.
   From Assumption \ref{as:UI}, $\left(Z_N^2\right)_{N > n}$ is in $L_1$, so its conditional expectation is uniformly integrable. For the second part, we highlight that
$\left(\mathcal{I}(\theta_{N-1})^{-1}\right)_{N > n}$ is in $L_1$, so there exists some $K< \infty$ such that $\sup_N E\left\{\left|\mathcal{I}(\theta_{N-1})^{-1}\right|\right\}\leq K$, giving us \begin{align*}
       E\left(\theta_N^2\right) &= \sum_{i = n+1}^N i^{-2} E\left\{\mathcal{I}(\theta_{N-1})^{-1}\right\}
        \leq K \sum_{i = n+1}^N i^{-2} < \infty.
   \end{align*}

\setcounter{equation}{0}
\setcounter{theorem}{0}
\setcounter{proposition}{0}
\setcounter{lemma}{0}
\setcounter{corollary}{0}
\setcounter{assumption}{0}
\section{Proof of Theorem \ref{th:pred_CLT}}\label{sec:app_predclt}
\subsection{Main theorem}
In this section, we cover the proof of Theorem \ref{th:pred_CLT} in detail. First, we require a specialised version of the martingale central limit theorem for our purpose based on Theorem \ref{app:th:tailsum_CLT}.
\begin{theorem}\label{app:th:tailsum_CLT_final}
Let $\left(S_N, \mathcal{F}_N\right)_{N \geq 1}$ be a martingale bounded in $L_2$ with martingale differences $X_i$, and let $\mathcal{F}_\infty = \sigma\left(\cup_{N\geq 1}\mathcal{F}_N\right)$. Furthermore, let $V_N^2$ and $s_N^2$ be defined as in (\ref{app:eq_condit_var}).
Assume that for all $\varepsilon > 0$,
 \begin{align} \label{app:tailsum_as1_final} 
s_N^{-2}\sum_{i = N+1}^\infty E\left\{X_i^2 \mathbbm{1}\left(X_i^2> \varepsilon^2 s_N^2\right)\right\} \to 0,
    \end{align}
    and 
    \begin{align} \label{app:tailsum_as2_final} 
        s_N^{-2}V_N^2 \to \eta^2
    \end{align}
    in probability, where $\eta^2$ is some random variable and $\eta^2 >0$ almost surely. Under the above assumptions, we have that
    \begin{align*}
       s_N^{-1}\sum_{i = N+1}^\infty X_i\to\mathcal{N}(0,\eta^2)
    \end{align*}
    $\mathcal{F}_\infty$-stably, and furthermore we have
\begin{align}\label{app:tailsum_res1_final}
   V_N^{-1} \sum_{i = N+1}^\infty X_i\to\mathcal{N}(0,1)
\end{align}
$\mathcal{F}_\infty$-mixing.
\end{theorem}
\begin{proof}
This is a direct application of Theorem \ref{app:th:tailsum_CLT}, where we use Lemma \ref{app:lem:uncondit_lind} to replace the assumption in  (\ref{app:tailsum_as1}) with (\ref{app:tailsum_as3}).
For the final statement (\ref{app:tailsum_res1_final}), as the convergence in (\ref{app:tailsum_res1}) is stable and $\eta^2 > 0$ almost surely by assumption, we can apply the same random norming argument as in \citet[Corollary 6.3]{Hausler2015} to show that
\begin{align*}
    V_N^{-1}\sum_{i =N+1}^\infty X_i \to \mathcal{N}(0,1)
\end{align*}
$\mathcal{F}_\infty$-mixing. 
\end{proof}
For a fixed dataset $Y_{1:n}$ and corresponding deterministic $\theta_n$, we define $S_N = \theta_N$ for $N \geq n$. Theorem \ref{th:pred_CLT} then follows if we can verify (\ref{app:tailsum_as1_final}) and (\ref{app:tailsum_as2_final}), where $\mathcal{F}_\infty$-mixing in (\ref{app:tailsum_res1_final}) implies convergence in distribution. As stated in the main paper, all of the following results will still hold if we replace $N^{-1}$ with a general sequence $N^{-1}$ which satisfies (\ref{eq:alpha}).

\subsection{Variance condition}
We first show the variance condition (\ref{app:tailsum_as2_final}), which is simpler.

\begin{lemma}\label{app:lem:var}
    Under Assumption \ref{as:UI}, for $ r_N^2 = \sum_{i = N+1}^\infty i^{-2}$, we have that
\begin{align}\label{app:eq:VN}
r_N^{-2}\,{V_N^2} \to \mathcal{I}(\theta_\infty)^{-1}
\end{align}
almost surely, and
\begin{align}\label{app:eq:SN}
r_N^{-2} s_N^2 \to E\left\{\mathcal{I}(\theta_\infty)^{-1}\right\}.
\end{align}
    Together, this gives
    \begin{align}\label{app:eq:prop_ratio}
        \frac{V_N^2}{s_N^2}\to \eta^2 =  \frac{\mathcal{I}(\theta_\infty)^{-1}}{E\left\{\mathcal{I}(\theta_\infty)^{-1}\right\}}
    \end{align}
    almost surely, where $\eta^2 > 0$ almost surely.
\end{lemma}
\begin{proof}
   Under the parametric martingale posterior, we have
   that
   \begin{align*}
      V_N^2 = \sum_{i = N+1}^\infty i^{-2}\mathcal{I}(\theta_{i-1})^{-1}, \quad   s_N^2 = \sum_{i = N+1}^\infty i^{-2}E\left\{\mathcal{I}(\theta_{i-1})^{-1}\right\}.
   \end{align*}
First we show that $\mathcal{I}(\theta_\infty)^{-1}$ exists and is in $L_1$. As $\theta_N \to \theta_\infty <\infty$ almost surely  and $\mathcal{I}(\theta)^{-1}$ is continuous for all $\theta \in \R$, the continuous mapping theorem gives 
\begin{align*}
    \mathcal{I}(\theta_N)^{-1} \to \mathcal{I}(\theta_\infty)^{-1}
\end{align*}
almost surely. As $\theta_\infty$ is finite almost surely, we have that $0<\mathcal{I}(\theta_\infty)^{-1} < \infty$ almost surely by Assumption \ref{as:existence}. Secondly, we have
\begin{align*}
    E\left\{\left| \mathcal{I}(\theta_\infty)^{-1}\right|\right\} &=  E\left\{\lim_{N \to \infty}\left| \mathcal{I}(\theta_N)^{-1}\right|\right\} \leq \liminf_{N\to \infty}E\left\{\left| \mathcal{I}(\theta_N)^{-1}\right|\right\}< \infty
\end{align*}
where we used Fatou's lemma for the first inequality and Lemma \ref{lem:L2} for the second, as $\left(\mathcal{I}(\theta_{N-1})^{-1}\right)_{N >n}$ being uniformly integrable implies $\sup_N E\left\{\left| \mathcal{I}(\theta_N)^{-1}\right|\right\} < \infty$.

Now suppose a convergent sequence $x_1,x_2,\ldots$ has limit $x$. For any $\varepsilon > 0$, there exists $M$ such that $|x_i - x|< \varepsilon$ for all $i \geq M$. Now choose $N > M$, and we see that
\begin{align*}
    \left|r_N^{-2}{\sum_{i = N+1}^\infty i^{-2} x_i} -  x\right| =   \left|\frac{\sum_{i = N+1}^\infty i^{-2} (x_i-x)}{\sum_{i = N+1}^\infty i^{-2}}\right| 
    \leq\frac{\sum_{i = N+1}^\infty i^{-2}   \left|x_i-x\right| }{\sum_{i = N+1}^\infty i^{-2}}\leq   \varepsilon.
\end{align*}
In the $V_N^2$ case, we have $\mathcal{I}(\theta_N)^{-1}\to \mathcal{I}(\theta_\infty)^{-1}$ almost surely, which gives
(\ref{app:eq:VN}). For the $s_N^2$ case, uniform integrability gives us $E\left\{\left|\mathcal{I}(\theta_N)^{-1}\right| \right\} \to E\left\{\mathcal{I}(\theta_\infty)^{-1}\right\}$, and we can apply the above to get (\ref{app:eq:SN}). Together, this then implies (\ref{app:eq:prop_ratio}).
 \end{proof}

\subsection{Lindeberg condition}
We next consider the Lindeberg condition, which is more involved.
\begin{lemma}\label{app:lem:lind}
    Under Assumption \ref{as:UI}, the Lindeberg condition (\ref{app:tailsum_as1_final}) holds.
\end{lemma}
\begin{proof}
Again, let $r_N^{2} = \sum_{i = N+1}^\infty i^{-2}$.
As $\left(Z_i^2\right)_{i > n}$ is uniformly integrable, for each $\delta >0$, there exists $K\geq 0$ such that 
\begin{align}\label{app:eq_UI}
\sup_{i> n} E\left\{Z_i^2 \mathbbm{1}\left(Z_i^2 > K \right)\right\} < \delta.
\end{align}
First, we have from Lemma \ref{app:lem:var} that $r_N^{-2} s_N^2 \to D < \infty$ for some positive constant $D$, so for some  small $\kappa >0$ and sufficiently large $N$, we have
\begin{align}\label{app:eq_sn_bounds}
r_N^{-2}\,{s_N^2} \geq D - \kappa, \quad {r_N^{2}}\,{s_N^{-2}} \leq D^{-1} + \kappa. 
\end{align}
eventually. We can then write
 (\ref{app:tailsum_as1_final})  as
 \begin{align*}
&r_N^{-2}\frac{r_N^2}{s_N^2}\sum_{i = N+1}^\infty E\left[X_i^2 \mathbbm{1}\left\{X_i^2> \varepsilon^2 r_N^2\left(\frac{s_N^2}{r_N^2}\right)\right\}\right]\\
&\leq r_N^{-2}\left(D^{-1} + \kappa\right)\sum_{i = N+1}^\infty E\left[X_i^2 \mathbbm{1}\left\{X_i^2> \varepsilon^2 (D- \kappa) \sum_{j = N+1}^\infty j^{-2} \right\}\right]\\
&= r_N^{-2}\left(D^{-1} + \kappa\right)\sum_{i = N+1}^\infty i^{-2}E\left[Z_i^2 \mathbbm{1}\left\{Z_i^2>  \varepsilon^2(D- \kappa) \, i^2 r_N^2\right\}\right],
    \end{align*}
    where the first inequality uses $E\left\{|Y| \mathbbm{1}\left(|Y| > a \right)\right\} \leq E\left\{|Y| \mathbbm{1}\left(|Y| > b \right)\right\}$ if $a \geq b$.

To simplify the above, we can use the fact that  $i^2 \geq (N+1)^2$ for $i > N$ and 
\begin{align*}
  (N+1)^{-1}\leq  r_N^2 \leq N^{-1}
\end{align*}
from an integral test argument. This then gives $i^2 r_N^2 \geq N+1$, which gives
an upper bound of  (\ref{app:tailsum_as1_final}) as
\begin{align*}
 r_N^{-2}{\left(D^{-1} + \kappa\right)}\sum_{i = N+1}^\infty i^{-2} E\left[Z_i^2 \mathbbm{1}\left\{Z_i^2>  \varepsilon^2(D- \kappa) (N+1)\right\}\right].
\end{align*}

Now pick an arbitrary $\varepsilon > 0$ and $\delta > 0$.
From (\ref{app:eq_UI}), we can choose $K$ sufficiently large so that
\begin{align*}
    \sup_{i > n}E\left\{Z_i^2 \mathbbm{1}\left(Z_i^2 > \varepsilon^2 K \right)\right\} < {\delta}\cdot
\end{align*}
We can further choose $N^*$ sufficiently large so that (\ref{app:eq_sn_bounds}) holds and $\left(D- \kappa\right)\left(N+1\right) > K
$ for all $N > N^*$. This then gives us
$$
\sup_{i> n} E\left[Z_i^2 \mathbbm{1}\left\{Z_i^2>  \varepsilon^2(D- \kappa) (N+1)\right\}\right] < \delta
$$
for all $N > N^*$. We thus have 
$$
 r_N^{-2} \left(D^{-1} + \kappa\right) \sum_{i = N+1}^\infty i^{-2} E\left[Z_i^2 \mathbbm{1}\left\{Z_i^2>  \varepsilon^2(D- \kappa) (N+1)\right\}\right]< \delta \left(D^{-1} + \kappa\right)
$$
for all $N > N^*$.

In other words, for each $\delta > 0$,  once can choose $N^*$ such that for all $N >N^*$, we have
\begin{align*}
    s_N^{-2}\sum_{i = N+1}^\infty E\left\{X_i^2 \mathbbm{1}\left(X_i^2 >  \varepsilon^2 s_N^2\right)\right\} <\left(D^{-1} + \kappa\right)\delta.
\end{align*}
Since $D^{-1} + \kappa$ is just a constant, we have the desired unconditional Lindeberg condition as the above term converges to 0 with $N\to \infty$ for each $\varepsilon > 0$. 
\end{proof}

\setcounter{equation}{0}
\setcounter{theorem}{0}
\setcounter{proposition}{0}
\setcounter{lemma}{0}
\setcounter{corollary}{0}
\setcounter{assumption}{0}
\section{Proof of Theorem \ref{th:bvm}}\label{sec:app_bvm}
\subsection{Prerequisites}
We now state our specialized martingale central limit theorem for one-sided martingale arrays.
\begin{theorem}\label{app:th:array_clt_final}
Let $\left(S_{ni}, \mathcal{F}_{ni}; i \geq 1, n \geq 1\right)$ be a square-integrable infinite martingale array, with martingale differences $X_{ni}$ which satisfy (\ref{app:infin_array_as1}). Furthermore, let $V_n^2$ and $s_n^2$ be defined according to (\ref{app:eq_condit_var_n}). 
Assume further that for all $\varepsilon >0$,
 \begin{align} \label{app:infin_array_as4} 
s_{n}^{-2}\sum_{i = 1}^{\infty} E\left\{X_{ni}^2 \mathbbm{1}\left(X_{ni}^2> \varepsilon^2s_n^2 \right)\right\} \to 0,
\end{align}
and
    \begin{align} \label{app:infin_array_as5} 
      s_n^{-2} V_{n}^2 \to \eta^2
\end{align}
in probability, where $\eta^2 > 0$ is a positive constant almost surely. Under the above assumptions, we have that
\begin{align*}
    V_{n}^{-1} \sum_{i = 1}^{\infty} X_i \to \mathcal{N}(0,1).
\end{align*}
in distribution.
\end{theorem}
\begin{proof}
Define the one-sided infinite martingale array $\left(T_{ni}, \mathcal{G}_{ni}, i \geq 1, n \geq 1\right)$ where
         \begin{align*}
             T_{ni} = s_n^{-1}\sum_{j = 1}^i X_{nj}, \quad \mathcal{G}_{ni} = \mathcal{F}_{ni}.
         \end{align*}
         We then have
         \begin{align*}
    \sup_{ni}E\left(T_{ni}^2\right) = \sup_{n}s_n^{-2} \sup_i\sum_{j = 1}^i X_{nj}^2 \leq \sup_n s_n^{-2} \, s_n^2 = 1 < \infty.
         \end{align*}
         Furthermore, a simple adaptation of Lemma \ref{app:lem:uncondit_lind} gives us that (\ref{app:infin_array_as4}) implies
         \begin{align*}  
s_{n}^{-2}\sum_{i = 1}^{\infty} E\left\{X_{ni}^2 \mathbbm{1}\left(X_{ni}^2> \varepsilon^2s_n^2 \right)\mid \mathcal{F}_{n,i-1}\right\} \to  0
\end{align*}
in probability, for all $\varepsilon > 0$.
We can thus apply Theorem \ref{app:th:infinite_array_clt} to $\left(T_{ni}, \mathcal{G}_{ni}, i \geq 1, n \geq 1\right)$. 
\end{proof}

We now prove Theorem \ref{th:bvm}. Consider the parametric martingale posterior starting at $\theta_n$ with sample size $n$, that is for $i \geq 1$:
\begin{align*}
    \widetilde{Y}_{ni}  \sim p_{\theta_{n,i-1}}, \quad 
    \theta_{ni} = \theta_{n,i-1} + (n+i)^{-1}Z_{ni}
\end{align*}
 where
 \begin{align*}
     Z_{ni} = \mathcal{I}\left( \theta_{n,i-1}\right)^{-1} s\left(\theta_{n,i-1}, \widetilde{Y}_{ni}\right)
 \end{align*}
 and  $\theta_{n0} = \theta_n$ and $\widetilde{Y}$ indicates simulated data. Define $S_{ni} = \theta_{ni}$ and $\mathcal{F}_{ni} = \sigma\left(\widetilde{Y}_{n1},\ldots,\widetilde{Y}_{ni}\right)$.

 We will verify the conditions of Theorem \ref{app:th:array_clt_final}. For the remainder of this section, let $U$ denote a neighborhood of $\theta^*$ where $\mathcal{I}(\theta)^{-1}$ is continuously differentiable, and let $m$ choose be chosen  sufficiently large so that $\theta_m$ lies within the neighborhood $U$, which exists from  Assumption \ref{as:bvm}. Furthermore, let $r_n^{2} = \sum_{i = 1}^\infty (n+i)^{-2} = \sum_{i = n+1}^\infty i^{-2}$. As before, we can show that
 \begin{align*}
 (n+1)^{-1} \leq   r_n^2 \leq  n^{-1}
 \end{align*}
using an integral test argument.

\subsection{Uniform integrability condition}
In this section, we derive the extension of Proposition \ref{prop:Z4} to the Bernstein-von Mises case, given below.
\begin{proposition}\label{app:prop:Z4_bvm}
      Let $Z(\theta,Y) = \mathcal{I}(\theta)^{-1}s(\theta, Y)$, where $Y \sim p_\theta$. Suppose there exists non-negative constants $B,C< \infty$ such that the following holds for all $\theta \in \Theta \subseteq \R$,
    \begin{align*}
    E_\theta\left\{Z(\theta,Y)^4\right\} \leq B + C \theta^4.
    \end{align*}
Furthermore, let $(\theta_n)_{n \geq 1}$ be a deterministic convergent sequence. Then $\left(Z^2_{ni}\right)_{i\geq 1, n \geq 1}$ is uniformly integrable. 
\end{proposition}
\begin{proof}
The proof of Proposition \ref{prop:Z4} can be easily applied to get 
\begin{align*}
    \sup_{i \geq 1} E\left(Z^4_{ni}\right) \leq C \sup_{i \geq 1}E\left(\theta_{ni}^4\right) + B &\leq D_1\theta_n^4  + D_2 \sum_{i = n+1}^\infty i^{-2}\\
    &\leq D_1\theta_n^4  + D_2n^{-1} 
\end{align*}
for all $n\geq 1$, where $D_1,D_2< \infty$ are non-negative finite constants which can be chosen independently of $n$. This gives
\begin{align*}
   \sup_{n \geq 1} \sup_{i \geq 1} E\left(Z^4_{ni}\right) \leq    D_1\sup_{n \geq 1}  \theta_n^4  + D_2n^{-1} < \infty
\end{align*}
where finiteness follows as $\theta_n$ is convergent. We thus have that 
$\left(Z^2_{ni}\right)_{i\geq 1, n \geq 1}$ is uniformly integrable.
\end{proof}

\subsection{Variance condition}
We now check the variance condition (\ref{app:infin_array_as5}). First, we show a uniform version of Lemma \ref{lem:L2}.
\begin{lemma}\label{app:lem:sn_UI}
Let $(\theta_n)_{n \geq 1}$ be a deterministic convergent sequence. Under Assumption \ref{as:bvm}, we have that
the collection $\left(\mathcal{I}(\theta_{ni})^{-1}:i \geq 1, n \geq 1\right)$ is uniformly integrable. Furthermore, we have $s_n^2 \to 0$.  
\end{lemma}
\begin{proof}
    We begin with
    \begin{align*}
        E\left(Z^2_{ni}\mid \mathcal{F}_{n,i-1}\right) = \mathcal{I}(\theta_{n,i-1})^{-1}.
    \end{align*}
    Since $\left(Z_{ni}^2: i \geq 1, n \geq 1\right)$ is uniformly integrable, there exists $L>0$ such that
    \begin{align*}
        \sup_{i\geq 1, n \geq 1}E\left(Z_{ni}^2 \right) \leq L.
    \end{align*}
    This then gives
    \begin{align*}
       s_n^2 &= \sum_{i = 1}^\infty (n+i)^{-2}E\left(Z_{ni}^2\right)\leq L r_n^2 \leq  Ln^{-1}.
    \end{align*}
    We thus have $s_n^2 \to 0$. For the uniform integrability result, it is an extension of the usual argument on conditional expectations (e.g. \citet[Exercise 13.3]{Williams1991}). 
    Since $\mathcal{F}_{ni}$ is a sub-$\sigma$-algebra of $\mathcal{F}$, and the set $\left(Z_{ni}^2:i \geq 1, n \geq 1\right)$ is uniformly integrable, the collection $\left(E\left(Z_{ni}^2\mid \mathcal{F}_{n,i-1}\right): i \geq 1, n \geq 1\right)$ is uniformly integrable.

\end{proof}

We will also need the following useful lemma, which is related to posterior consistency.
\begin{lemma}\label{app:lem:L2}
      Let $(\theta_n)_{n \geq 1}$ be a deterministic  sequence which converges to $\theta^*$. Under Assumption \ref{as:bvm}, we have that 
             \begin{align*}
           E\left\{(\theta_{n\infty} - \theta^*)^2 \right\} \to 0.
       \end{align*}
       Furthermore, we have
       \begin{align*}
           \sup_{i \geq 1}E\left\{(\theta_{ni} - \theta^*)^2\right\} \to 0.
       \end{align*}
\end{lemma}
\begin{proof}
   For each $n$, we have that $\theta_{n\infty}$ exists almost surely as the sequence $\left(\theta_{ni}\right)_{i \geq 1}$ is a martingale bounded in $L_2$.
We can expand the expectation as
\begin{align*}
 E\left\{(\theta_{n\infty} - \theta^*)^2\right\} = E\left\{(\theta_{n\infty} - \theta_{n0})^2\right\} + (\theta_{n0} - \theta^*)^2 + 2 E\left(\theta_{n\infty} - \theta_{n0}\right)\left(\theta_{n0} - \theta^*\right)
\end{align*}
As each martingale is bounded in $L_2$ (and thus uniformly integrable), we have $E\left(\theta_{n\infty} \right)= \theta_{n0}$ so the last term is 0. We can also highlight that
\begin{align*}
    E\left\{(\theta_{n\infty} - \theta_{n0})^2\right\} = s_n^2.
\end{align*}
We thus have $E\left\{(\theta_{n\infty} - \theta^*)^2\right\} \to 0$ 
 as $s_n^2 \to 0$ by Lemma \ref{app:lem:sn_UI} and $\theta_{n0} \to \theta^*$ by assumption.

 Furthermore, we have 
 \begin{align*}
     E\left\{(\theta_{ni}- \theta^*)^2\right\}  \leq      E\left\{(\theta_{n\infty}- \theta^*)^2\right\}, 
 \end{align*}
 so again we have $\sup_{i \geq 1}E\left\{(\theta_{ni}- \theta^*)^2\right\} \to 0$.
\end{proof}

We can now show the variance condition (\ref{app:infin_array_as5}).
\begin{lemma}\label{app:lem:var_bvm}
    Let $(\theta_n)_{n \geq 1}$ be a deterministic  sequence which converges to $\theta^*$. Under Assumption \ref{as:bvm}, we have that
      \begin{align*}
    {r_n^{-2}V_n^2} \to  \mathcal{I}(\theta^*)^{-1}
    \end{align*}
    in probability, and  
    \begin{align*}
        \frac{V_n^2}{s_n^2} \to 1
    \end{align*}
    in probability.
  \end{lemma}
\begin{proof}
      As before, we have that
     \begin{align*}
      V_n^2 = \sum_{i = 1}^\infty (n+i)^{-2}\, \mathcal{I}(\theta_{n,i-1})^{-1}, \quad   s_n^2 = \sum_{i = 1}^\infty (n+i)^{-2}E\left\{\mathcal{I}(\theta_{n,i-1})^{-1}\right\}.
   \end{align*}
However, this case is more complex than Lemma \ref{app:lem:var}. To see this, although we have $\lim_{n \to \infty}\mathcal{I}(\theta_{n0})^{-1} =\mathcal{I}(\theta^*)^{-1}$, we require this convergence to be uniform over all values of $i$.
   
We start by showing the following:
   \begin{align}\label{app:eq:sn_bvm}
       r_n^{-2}{s_n^2}\to \mathcal{I}(\theta^*)^{-1}.
   \end{align}
We can write
   \begin{align*}
       \left| r_n^{-2}{s_n^2}- \mathcal{I}(\theta^*)^{-1}\right| \leq r_n^{-2}{\sum_{i = 1}^\infty (n+i)^{-2} E\left\{\left|\mathcal{I}(\theta_{n,i-1})^{-1} - \mathcal{I}(\theta^*)^{-1}\right|\right\}}.
   \end{align*}
We now upper bound the expectation term in the above summation. 
For some $r \in \mathbb{R}^+$, consider a compact interval $ I = [\theta^* - r, \theta^*+r]$ where $I \subset U$. 
Now consider the term
\begin{align*}
    \sup_{i \geq 1} E\left\{\left| \mathcal{I}(\theta_{ni})^{-1} - \mathcal{I}(\theta^*)^{-1}\right|\right\} &\leq  \sup_{i \geq 1} E\left\{\left| \mathcal{I}(\theta_{ni})^{-1} - \mathcal{I}(\theta^*)^{-1}\right|\mathbbm{1}\left(\theta_{ni} \in I\right)\right\}\\
    &+ \sup_{i \geq 1} E\left\{\left| \mathcal{I}(\theta_{ni})^{-1} - \mathcal{I}(\theta^*)^{-1}\right|\mathbbm{1}\left(\theta_{ni} \in I^C\right)\right\},
\end{align*}
where we have split the expectation into $\theta_{ni}$ lying inside or outside $I$. The proof technique will involve bounding the first term using the continuous differentiability of $\mathcal{I}(\theta)^{-1}$, and bounding the second term using the uniform integrability of $\left(\mathcal{I}(\theta_{ni})^{-1}: i \geq 1, n \geq 1\right)$.

For the first term on the right of the inequality, 
within this compact interval, $\mathcal{I}(\theta)^{-1}$ is continuously differentiable from Assumption \ref{as:bvm} and is thus $K$-Lipschitz  for some positive constant $K$. 
We thus have
\begin{align*}
    \sup_{i \geq 1} E\left\{\left| \mathcal{I}(\theta_{ni})^{-1} - \mathcal{I}(\theta^*)^{-1}\right|\mathbbm{1}\left(\theta_{ni} \in I\right)\right\} \leq  K \sup_{i \geq 1} E\left(\left| \theta_{ni} - \theta^*\right|\right)\to 0
\end{align*}
as we have $\sup_{i \geq 1}E\left\{(\theta_{ni} -\theta^*)^2\right\} \to 0$ from Lemma \ref{app:lem:L2}.

For the second term, we will use uniform integrability of $\left(|\mathcal{I}(\theta_{ni})^{-1}- \mathcal{I}(\theta^*)^{-1}|: i \geq 1, n \geq 1\right)$, which follows from Lemma \ref{app:lem:sn_UI} and the triangle inequality.  
For each $\varepsilon > 0$, there exists $\delta > 0$ such that 
\begin{align*}
    E\left\{\left|\mathcal{I}(\theta_{ni})^{-1} - \mathcal{I}(\theta^*)^{-1} \right| \mathbbm{1}(\theta_{ni} \in A)\right\} \leq \varepsilon
\end{align*}
for all $n \geq 1, i \geq 1$ whenever $P(\theta_{ni}\in A) \leq \delta$. We are thus interested in applying this  with $A = I^C$.
 For all $n \geq 1$ and $i\geq 1$, we have that the above is upper bounded by $\varepsilon$ as long as $P\left(\theta_{ni} \in I^C\right)\leq \delta$. 
Now we use Markov's inequality to upper bound $P\left(\theta_{ni}\in I^C\right)$:
\begin{align*}
    P\left(\theta_{ni} \in I^C\right) &= P(\left| \theta_{ni} - \theta^*\right|\geq r)\\
    &\leq r^{-2}E\left\{(\theta_{ni}-\theta^*)^2\right\}.
\end{align*}
From Lemma \ref{app:lem:L2}, we have that
\begin{align*}
    \sup_{i \geq 1}P\left(\theta_{ni}\in A^C\right) \leq r^{-2}\sup_{i \geq 1} E\left\{(\theta_{ni}- \theta^*)^2\right\}\to 0.
\end{align*}
We can thus choose $n^*\geq 1$ sufficiently large such that
\begin{align*}
    \sup_{i \geq 1} P\left(\theta_{ni} \in I^C\right) \leq \delta
\end{align*}
for all $n > n^*$.
In other words, we have $P\left(\theta_{ni} \in I^C\right)  \leq \delta$ for all $n \geq n^*$ and all $i \geq 1$, so 
\begin{align*}
        E\left\{\left|\mathcal{I}(\theta_{ni})^{-1} - \mathcal{I}(\theta^*)^{-1} \right| \mathbbm{1}\left(\theta_{ni} \in I^C\right)\right\} \leq \varepsilon.
\end{align*}
for all $i \geq 1$ and $n \geq n^*$. Since $\varepsilon >0$ was arbitrary, this gives 
\begin{align*}
     \sup_{i \geq 1} E\left\{\left|\mathcal{I}(\theta_{ni})^{-1} - \mathcal{I}(\theta^*)^{-1} \right| \mathbbm{1}\left(\theta_{ni} \in I^C\right)\right\} \to 0.
\end{align*}
Putting this together yields 
$\sup_{i \geq 1} E\left\{\mathcal{I}(\theta_{ni})^{-1} - \mathcal{I}(\theta^*)\right\} \to 0$.
This thus gives
\begin{align*}
    r_n^{-2}{\sum_{i = 1}^\infty (n+i)^{-2} E\left(\left|\mathcal{I}(\theta_{n,i-1})^{-1} - \mathcal{I}(\theta^*)^{-1}\right|\right)} &\leq \sup_{i \geq 1} E\left\{\left|\mathcal{I}(\theta_{ni})^{-1} - \mathcal{I}(\theta^*)^{-1}\right|\right\}\to 0,
\end{align*}
which gives  $r_n^{-2} s_n^2 \to \mathcal{I}(\theta^*)^{-1}$.

We now show that the above implies $r_n^{-2}V_n^2 \to \mathcal{I}(\theta^*)^{-1}$ in probability. For any $\varepsilon > 0$, Markov's inequality gives
\begin{align*}
P\left(\left|r_n^{-2}V_n^2 - \mathcal{I}(\theta^*)^{-1} \right| > \varepsilon\right) &=     P\left(\left|r_n^{-2}{\sum_{i = 1}^\infty (n+i)^{-2} \mathcal{I}(\theta_{n,i-1})^{-1}}- \mathcal{I}(\theta^*)^{-1}\right| > \varepsilon\right) \\
&\leq \frac{E\left\{\left|r_n^{-2}{\sum_{i = 1}^\infty (n+i)^{-2} \mathcal{I}(\theta_{n,i-1})^{-1}}- \mathcal{I}(\theta^*)^{-1}\right| \right\}}{\varepsilon}\\
&\leq \frac{r_n^{-2}\sum_{i = 1}^\infty (n+i)^{-2}  E\left\{\left|\mathcal{I}(\theta_{n,i-1})^{-1}- \mathcal{I}(\theta^*)^{-1}\right| \right\}}{\varepsilon} \to 0
\end{align*}
where we have applied Tonelli's theorem and the earlier proof. Finally, the continuous mapping theorem gives us $V_n^2/ s_n^2 \to 1$ in probability.

\end{proof}
\subsection{Lindeberg condition}
We now show the Lindeberg condition (\ref{app:infin_array_as4}).
\begin{lemma}
    Let $\theta_n$ be a deterministic sequence which converges to $\theta^*$ as $n \to \infty$. Under Assumption \ref{as:bvm}, the Lindeberg condition (\ref{app:infin_array_as4}) holds.
\end{lemma}
\begin{proof}
The proof follows in the same way as Lemma \ref{app:lem:lind}, so we omit some details for brevity.  
 From (\ref{app:eq:sn_bvm}) in the proof of Lemma \ref{app:lem:var_bvm}, we have that
    $s_n^2 \approx Dn^{-1}$, which can be made rigorous using the same argument as Lemma \ref{app:lem:lind}. 
    
    For some $\varepsilon > 0$, we can write (\ref{app:infin_array_as4}) as
    \begin{align*}
        s_{n}^{-2}\sum_{i = 1}^{\infty} E\left\{X_{ni}^2 \mathbbm{1}\left(X_{ni}^2> \varepsilon^2s_n^2 \right)\right\} \approx D^{-1}n\sum_{i = 1}^{\infty} (n+i)^{-2} E\left[ Z^2_{ni} \, \mathbbm{1}\left\{Z_{ni}^2> \varepsilon^2 Dn^{-1} (n+i)^{2} 
 \right\}\right].
    \end{align*}
Since $\left(Z_{ni}^2: i \geq 1, n \geq 1\right)$ is uniformly integrable, we have that for each $\delta >0$ there exists $K$ (chosen independently of $n$) such that
\begin{align*}
    \sup_{n \geq 1}\sup_{i\geq 1}E\left\{Z_{ni}^2 \mathbbm{1}\left(Z_{ni}^2 > \varepsilon^2 K\right)\right\} < \delta.
\end{align*}
Now choose $n^*$ sufficiently large so $Dn^* > K$. Then for all $n > n^*$, we have
\begin{align*}
   \sup_{i\geq 1} E\left[Z_{ni}^2 \mathbbm{1}\left\{Z_{ni}^2 > \varepsilon^2 Dn^{-1}(n+i)^2\right\}\right]  \leq     \sup_{i\geq 1} E\left\{Z_{ni}^2 \mathbbm{1}\left(Z_{ni}^2 > \varepsilon^2 Dn\right)\right\}< \delta.
\end{align*}
Then we have
\begin{align*}
    D^{-1} n \sum_{i = 1}^\infty  (n+i)^{-2} E\left[Z_{ni}^2 \mathbbm{1}\left\{Z_{ni}^2 > \varepsilon^2 Dn^{-1}(n+i)^2\right\}\right] \leq \delta D^{-1} n \sum_{i = 1}^\infty (n+i)^{-2} \approx  \delta D^{-1}.
\end{align*}
Since $D$ is just a constant and $\delta > 0$ is arbitrary, we have convergence of the above to $0$ as $n \to \infty$.
\end{proof}

\subsection{Asymptotics of initial estimate}\label{sec:init_estimate}
One can treat (\ref{eq:sgd}) as an approximation to the maximum likelihood estimate. In this case, one can justify setting $\theta_n$ to the maximum likelihood estimate and using (\ref{eq:sgd}) for predictive resampling as the recursive update is asymptotically equivalent. At a small cost of loss of coherence, the maximum likelihood estimate would have the additional benefit of being invariant to the ordering of the data, and the asymptotic theory may be better understood.

The asymptotics of estimates $\theta_n$ arising from maximum likelihood estimation or stochastic approximation (\ref{eq:sgd}) are well understood. Conditions for weak consistency and asymptotic normality of maximum likelihood estimators can be found in \citet[Chapter 5]{van2000}, and additional assumptions needed for strong consistency can be found in the discussions following \citet[Theorem 2.1]{Newey1994} or the seminal work of \citet{Wald1949}. Strong consistency for stochastic approximation can be shown using the almost supermartingale convergence theorem of \cite{Robbins1971} or using ordinary differential equation methods \citep[Chapter 5]{Harold2003}. Asymptotic normality results for stochastic approximation can be found in \cite{Fabian1968} and \cite{Lai2003}.

\setcounter{equation}{0}
\setcounter{theorem}{0}
\setcounter{proposition}{0}
\setcounter{lemma}{0}
\setcounter{corollary}{0}
\setcounter{assumption}{0}

\section{Multivariate extension}\label{sec:app_mv}
\subsection{Setup and Assumptions}
In this section, we consider the multivariate case where $\theta \in \R^p$. Our recursive update for the multivariate case is then identical to (\ref{eq:sgd}), but $s(\theta,Y)$ is a $p$-vector, and $\mathcal{I}(\theta)$ is the $p \times p$ Fisher information matrix. We then have the multivariate version of Assumption \ref{as:existence} below.
\begin{assumption}\label{as:existence_mv}
For each $\theta \in \R^p$, we require the following, where expectations are understood to be over $Y \sim p_\theta$.
The score function $s(\theta,y)= \nabla_\theta\log p_\theta(y)$ exists, each component is finite and has mean zero, that is 
    $
    E_\theta\left\{s_j(\theta,Y) \right\} = 0
    $
   for $j = 1,\ldots,p$.
Furthermore, the score function is square-integrable, that is the Fisher information matrix,
    $\mathcal{I}(\theta) =  E_\theta\left\{s(\theta,Y) \, s(\theta,Y)^{T}\right\},$
    exists. Finally, we require $\mathcal{I}(\cdot)$ to be continuous at $\theta$ and for $p_\theta$ to be regular, that is $\mathcal{I}(\theta)$ is positive definite.
\end{assumption}

We will assume Assumption \ref{as:existence_mv} holds throughout. We now apply the respective martingale central limit theorems to the Cram\'{e}r-Wold device,
$$\zeta_N(t) = t^T \theta_N,$$ 
for each $t \in \R^p$ and $N > n$. We then have the update
\begin{align*}
    \zeta_N(t) &= \zeta_{N-1}(t) + N^{-1} \, t^T \mathcal{I}(\theta_{N-1})^{-1} \, s(\theta_{N-1},Y_N).
\end{align*}
It is clear that we have a martingale with respect to the filtration $\left(\mathcal{F}_N\right)_{N > n}$ for $\mathcal{F}_N = \sigma\left(Y_{n+1},\ldots,Y_N\right)$, as 
\begin{align*}
    E\left\{\zeta_N(t) \mid \mathcal{F}_{N-1}\right\} = \zeta_{N-1}(t).
\end{align*}
Let us define the update term
\begin{align*}
    Z(\theta,Y;t) &= t^T Z(\theta,Y),
\end{align*}
 where we have $E_{\theta}\left\{Z(\theta,Y;t)^2\right\} = t^T \mathcal{I}(\theta)^{-1}t$. 
 In order to apply the martingale central limit theorems,  
we require multivariate extensions of 
Assumption \ref{as:UI} and Proposition \ref{prop:Z4}, which we outline below.
 \begin{assumption}\label{as:UI_mv}
For each $t \in \R^p$, let $Z_N(t)$ denote the natural gradient of the Cram\'{e}r-Wold device:
\begin{align*}
    Z_N(t) =  t^T  \mathcal{I}(\theta_{N-1})^{-1}s(\theta_{N-1}, Y_N).
\end{align*}
The sequence $\left(Z^2_N(t)\right)_{N >n}$ is uniformly integrable under predictive resampling starting from the initial $\theta_n$.
\end{assumption}

\begin{proposition}\label{prop:Z4_mv}
    Let $Z(\theta,Y) = \mathcal{I}(\theta)^{-1}s(\theta, Y)$, where $Y \sim p_\theta$. Suppose there exists non-negative constants $B,C< \infty$ such that the following holds for all $\theta \in \Theta \subseteq \R^p$:
    \begin{align*}
    E_\theta\left\{\|Z(\theta,Y)\|^4\right\} \leq B + C \|\theta\|^4.
    \end{align*}
   Then $\left(Z^2_N(t)\right)_{N > n}$ is uniformly integrable for each $t \in \R^p$.
\end{proposition}
\begin{proof}
      As $Z(\theta,Y;t)^4=\left| \langle t,  Z(Y,\theta)\rangle\right|^4$, Cauchy-Schwarz gives
\begin{align}\label{app:eq_zt}
     E\left\{\left| \langle t,  Z(Y,\theta)\rangle\right|^4\right\} &\leq \|t\|^4 E\left\{\|Z(Y,\theta)\|^4\right\}\\
    & \leq  \|t\|^4\left(B + C\|\theta\|^4\right).\notag
\end{align}
We now seek to upper bound $B + C\|\theta_N\|^4$. For  $Z_N = Z(\theta_{N-1},Y_N)$, we have
\begin{align*}
    \|\theta_N\|^4 &= \left\{\|\theta_{N-1} + N^{-1} Z_N \|^2\right\}^2\\
    &= \left(  \|\theta_{N-1}\|^2 
 +N^{-2} \|Z_{N}\|^2  + N^{-1} \langle \theta_{N-1}, Z_N \rangle\right)^2\\
 &= \|\theta_{N-1}\|^4 + N^{-4} \|Z_N\|^4 + N^{-2} \langle \theta_{N-1}, Z_N \rangle^2 \\
 &+ 2\left(N^{-2}\|\theta_{N-1}\|^2\|Z_N\|^2 + N^{-1} \|\theta_{N-1}\|^2 \langle \theta_{N-1}, Z_N \rangle + N^{-3} \|Z_N\|^2 \langle \theta_{N-1}, Z_N \rangle\right).
    \end{align*}
From the martingale condition, we have
\begin{align*}
    E\left( \langle \theta_{N-1}, Z_N \rangle \mid \mathcal{F}_{N-1}\right) = 0.
\end{align*}
Together this gives
\begin{align*}
      E\left(\|\theta_N\|^4\mid \mathcal{F}_{N-1}\right) &= \|\theta_{N-1}\|^4 + N^{-4} E\left(\|Z_N\|^4\mid \mathcal{F}_{N-1}\right) + N^{-2}  E\left(\langle \theta_{N-1}, Z_N \rangle^2 \mid \mathcal{F}_{N-1}\right)\\
      &+ 2\left\{N^{-2} \|\theta_{N-1}\|^2 E\left(\|Z_N\|^2 \mid \mathcal{F}_{N-1}\right) + N^{-3} E\left(\|Z_N\|^2 \langle \theta_{N-1}, Z_N \rangle \mid \mathcal{F}_{N-1}\right)\right\}.
\end{align*}
We can now apply Cauchy-Schwarz to $\langle \theta_{N-1},Z_N\rangle$, which gives
\begin{align*}
 E\left(\|\theta_N\|^4\mid \mathcal{F}_{N-1}\right) &\leq    \|\theta_{N-1}\|^4 + N^{-4} E\left(\|Z_N\|^4\mid \mathcal{F}_{N-1}\right) + 3N^{-2}  \|\theta_{N-1}\|^2E\left(\|Z_N\|^2 \mid \mathcal{F}_{N-1}\right)\\
      &+ 2N^{-3} \|\theta_{N-1}\|E\left(\|Z_N\|^3  \mid \mathcal{F}_{N-1}\right)
\end{align*}
 Following an identical argument to the proof of Proposition \ref{prop:Z4} using H\"{o}lder's inequality,
one can show that
\begin{align*}
     E\left(\|\theta_N\|^4\mid \mathcal{F}_{N-1}\right) &\leq \|\theta_{N-1}\|^4 \left(1 + D_1  N^{-2}\right) + D_2  N^{-2}
\end{align*}
for some non-negative constants $D_1,D_2 < \infty$, which then gives \begin{align*}
    \sup_{N > n}E\left(\|Z_N\|^4\right) < \infty.
\end{align*}
We thus have $\sup_N Z^4_N(t) < \infty$ from (\ref{app:eq_zt}), which implies the uniform integrability of $\left(Z^2_N(t)\right)_{N> n}$ for each $t \in \R^p$.
\end{proof}
 We thus ensure that Assumption \ref{as:UI_mv} holds by controlling $\sup_N E\left\{Z^4_N(t)\right\}<\infty$ with a condition independent of $t$ using Proposition \ref{prop:Z4_mv}, which can be checked relatively easily.


\subsection{Predictive central limit theorem}
For the predictive central limit theorem, we will write the estimate of the limiting covariance matrix as
\begin{align}\label{eq:pred_as_cov}
    \widehat{\Sigma}_N = r_N^2\mathcal{I}(\theta_N)^{-1},
    \end{align}
    where once again we have $r_N^2 =\sum_{i = N+1}^\infty i^{-2}$.
Furthermore, we write $A^{1/2}$ for the principal square root of a positive-definite matrix $A$.
We then have the multivariate version of Theorem \ref{th:pred_CLT}.
\begin{theorem}\label{th:pred_CLT_mv}
    Under Assumptions \ref{as:existence_mv} and \ref{as:UI_mv}, the multivariate version of the parametric  martingale posterior arising from (\ref{eq:sgd}) satisfies
    \begin{align*}        \widehat{\Sigma}_N^{-\frac{1}{2}}\left(\theta_\infty - \theta_N\right) \to \mathcal{N}(0,I)
    \end{align*}
    $\mathcal{F}_\infty$-mixing, where $\widehat{\Sigma}_N$ is defined in (\ref{eq:pred_as_cov}). 
\end{theorem}
\begin{proof}
    As the proof is almost identical to Theorem \ref{th:pred_CLT} and Corollary \ref{corr:approx_variance}, we omit repeated details. For each $t \in \R^p$, let us additionally define $X_N(t) = \zeta_N(t) - \zeta_{N-1}(t)$ and 
$$
V_N^2(t) = \sum_{i=N+1}^\infty E\left\{X_i^2(t) \mid \mathcal{F}_{i-1}\right\}, \quad s_N^2(t) = \sum_{i=N+1}^\infty E\left\{X_i^2(t)\right\}. 
$$
We seek to apply Theorem \ref{app:th:tailsum_CLT_final} with $\zeta_N(t)$ as the martingale, so we must check that $\zeta_N(t)$ is bounded in $L_2$, and conditions (\ref{app:tailsum_as1_final}) and (\ref{app:tailsum_as2_final}) hold. We thus extend Lemmas \ref{lem:L2},  \ref{app:lem:var} and \ref{app:lem:lind} to the $\zeta_N(t)$ case.

As stated earlier, we have $E\left\{Z_N^2(t) \mid \mathcal{F}_{N-1}\right\} = t^{T} \mathcal{I}(\theta_{N-1})^{-1}t$. 
Following Lemma \ref{lem:L2}, we can show that $\left(t^{T} \mathcal{I}(\theta_{N-1})^{-1}t\right)_{N > n}$ is uniformly integrable and $\sup_{N}E\left\{\zeta_N^2(t)\right\} < \infty$ so $\left(\zeta_N(t)\right)_{N > n}$ is a martingale bounded in $L_2$. Next, we want to show that
\begin{align*}
    \frac{V_N^2(t)}{s_N^2(t)} \to \eta^2(t) = \frac{t^T \mathcal{I}(\theta_\infty)^{-1}t}{t^T E\left\{\mathcal{I}(\theta_\infty)^{-1}\right\} t}
\end{align*}
almost surely, where $\eta^2(t) > 0$ almost surely for $t \neq 0$. We carry out the same proof as Lemma \ref{app:lem:var} with the continuity of $t^T \mathcal{I}(\theta)^{-1}t$ in place of $\mathcal{I}(\theta)^{-1}$. Two useful intermediate results are
\begin{align*}
    r_N^{-2}s_N^2(t) \to t^T E\left\{\mathcal{I}(\theta_\infty)^{-1}\right\} t, \quad 
    r_N^{-2}{V_N^2(t)} \to t^T\mathcal{I}(\theta_\infty)^{-1}t
\end{align*}
where the second statement holds almost surely. Finally checking the Lindeberg condition is the same as Lemma \ref{app:lem:lind} but replacing $Z_i^2$ with $Z_i^2(t)$ and $s_N^2$ with $s_N^2(t)$. Since we only rely on uniform integrability of $\left(Z_N^2(t)\right)_{N > n}$ and convergence of $r_N^{-2}s_N^2(t)$, the proof can easily be extended to give (\ref{app:tailsum_as1_final}). Theorem \ref{app:th:tailsum_CLT_final} thus gives
\begin{equation*}
    V_N^{-1}(t) \left\{\zeta_\infty(t) - \zeta_N(t) \right\} \to\mathcal{N}(0,1)
\end{equation*}
$\mathcal{F}_\infty$-mixing. 

The final step, which is new for the multivariate case, is to move $V_N^2$ to the right in order to apply the stable version of Cram\'{e}r-Wold:
\begin{align*}
r_N^{-1}\left\{\zeta_\infty(t) - \zeta_N(t) \right\}  &=   r_N^{-1}V_N(t)  \left[V_N^{-1}(t) \left\{\zeta_\infty(t) - \zeta_N(t) \right\} \right]\\
&\to \mathcal{N}\left\{0,t^T\mathcal{I}(\theta_\infty)^{-1}t\right\}
\end{align*}
$\mathcal{F}_\infty$-stably, which can be shown using the stable version of Slutsky's theorem  (Proposition \ref{app:prop:slutsky_stable}).
 By the stable version of Cram\'{e}r-Wold (Proposition \ref{app:prop:cramwold_stable}), we then have 
\begin{align*}
    r_N^{-1}\left(\theta_\infty - \theta_N\right) \to \mathcal{N}\{0,\mathcal{I}(\theta_\infty)^{-1}\}
\end{align*}
$\mathcal{F}_\infty$-stably. 

As $\theta_N \to \theta_\infty$ almost surely, we have $\mathcal{I}(\theta_N)^{1/2} \to \mathcal{I}(\theta_\infty)^{1/2}$ almost surely from the continuity of $\mathcal{I}(\theta)$, where $A^{1/2}$ is the principal square root of a positive definite matrix $A$. 
We can now move the normalizer back to the left using stable convergence, following Propositions \ref{app:prop:cont_map_stable} and \ref{app:prop:slutsky_stable}:
\begin{align*}
     r_N^{-1}\mathcal{I}(\theta_N)^{{1}/{2}}\left(\theta_\infty - \theta_N\right)= \widehat{\Sigma}_N^{-\frac{1}{2}}\left(\theta_\infty - \theta_N\right)  \to \mathcal{N}(0,I)
\end{align*}
$\mathcal{F}_\infty$-stably, as desired. 
\end{proof}
\subsection{Bernstein-von Mises theorem}
For the Bernstein-von Mises result, we require the multivariate version of Assumption \ref{as:bvm}.
\begin{assumption}\label{as:bvm_mv}
  Suppose  $\mathcal{I}(\theta)^{-1}$ is continuously differentiable in a neighbourhood of $\theta^*$. Furthermore, assume the conditions of Proposition \ref{prop:Z4_mv} hold. 
\end{assumption}
We then have the main theorem below.
\begin{theorem}\label{th:bvm_as_mv}
     Suppose $\left(\theta_n\right)_{n \geq 1}$ is estimated from $Y_{1:n}$ drawn from $P^*$ and is strongly consistent at $\theta^*$.
     Under Assumptions \ref{as:existence_mv} and  \ref{as:bvm_mv}, we have
    \begin{align*}
       r_n^{-1}\left(\theta_{n\infty} - \theta_n\right) \to \mathcal{N}\left\{0,\mathcal{I}(\theta^*)^{-1}\right\}
    \end{align*}
    in distribution
    $P^*$-almost surely as $n \to \infty$, where $r_n^2 = {\sum_{i = n+1}^\infty i^{-2}}$.
\end{theorem}
\begin{proof}
    The multivariate case is a simple extension of Theorem \ref{th:bvm} again using the Cram\'{e}r-Wold device, where we do not need to worry about stable convergence as $\theta_n\to \theta^*$ for some deterministic value $\theta^*$. 
Under Assumption \ref{as:bvm_mv}, it is not too challenging to verify that for each $t \in \mathbb{R}$, we have that $t^T{\mathcal{I}(\theta)^{-1}}t$ is continuously differentiable in a neighborhood of $\theta^*$ and $\left(Z_{ni}^2\left(t\right)\right)_{i \geq 1, n \geq 1}$ is uniformly integrable, where
\begin{align*}
    Z_{ni}(t) = t^T \mathcal{I}(\theta_{n,i-1})^{-1} s\left(\theta_{ni,i-1}, \widetilde{Y}_{ni}\right).
\end{align*}
We can thus apply the exact same proof as Theorem \ref{th:bvm}, which we omit, giving the following.
\begin{theorem}\label{app:th:bvm_as_mv_cw}
    Suppose $\left(\theta_n\right)_{n \geq 1}$ is estimated from $Y_{1:n}$ drawn from $P^*$ and is strongly consistent at $\theta^*$.
    Under Assumptions \ref{as:existence_mv} and \ref{as:bvm_mv}, for each $t \in \mathbb{R}$, we have
    \begin{align*}
       r_n^{-1}t^T\left(\theta_{n\infty} - \theta_n\right) \to \mathcal{N}\left\{0,t^T\mathcal{I}(\theta^*)^{-1}t\right\}
    \end{align*}
    in distribution
    $P^*$-almost surely as $n \to \infty$, where $r_n^2 = {\sum_{i = n+1}^\infty i^{-2}}$.
\end{theorem}
A simple application of the (regular) Cram\'{e}r-Wold device then gives Theorem \ref{th:bvm_as_mv}. 
\end{proof}

\setcounter{equation}{0}
\setcounter{theorem}{0}
\setcounter{proposition}{0}
\setcounter{lemma}{0}
\setcounter{corollary}{0}
\setcounter{assumption}{0}
\section{Regression extension}\label{sec:app_reg}
\subsection{Setup and assumptions}
We now cover the extension of Theorems \ref{th:pred_CLT_mv} and \ref{th:bvm_as_mv}  to the regression setup, where we follow exactly the setup of the main paper. Let $X \in \mathcal{X} \subseteq \R^p$, and $Y \in \mathcal{Y}$ where $\mathcal{Y}\subseteq \R$ or $\mathcal{Y}= \{0,1\}$ in the continuous and binary case respectively, and let $p_{\theta}(y \mid x)$ denote the conditional model.  For predictive resampling, we will assume that $X_{n+1:\infty}$ is drawn independent and identically distributed from the empirical distribution of $X_{1:n}$, which we write as $p_n^x=n^{-1}\sum_{i = 1}^n \delta_{X_i}$. We will then study predictive resampling under the following update rule:
\begin{align}\label{eq:sgd_reg}
        \theta_{N} = \theta_{N-1} + N^{-1}\widehat{\mathcal{I}}_n(\theta_{N-1})^{-1} s(\theta_{N-1}, Y_{N}; X_{N}),
\end{align}
for $N > n$, where $\theta_n$ is the initial estimate. The case where the deterministic sequence $X_{1:\infty}$ and $\mathcal{I}(\theta)$ is known is then a simplification, which we do not cover here. To begin, we have the following assumption.
\begin{assumption}\label{as:existence_reg}
For each $\theta \in \R^p$ we require the following, where expectations are understood to be over $Y \mid X \sim p_\theta(\cdot \mid X)$ and $X \sim p_n^x$.
The score function $s(\theta,y;x)= \nabla_\theta\log p_\theta(y\mid x)$ exists, each component is finite and has mean zero, that is 
    $
    E_{\theta}\left\{s_j(\theta,Y;x) \right\} = 0
    $
   for each $x$ and  $j = 1,\ldots,p$.
Furthermore, the score function is square-integrable, that is the (estimated) Fisher information matrix,
\begin{align*}
    \widehat{\mathcal{I}}_n(\theta) =  \frac{1}{n}\sum_{i = 1}^n \mathcal{I}(\theta; X_i)
\end{align*}
    exists. Finally, we require $\widehat{\mathcal{I}}_n(\cdot)$ to be continuous at $\theta$ and $\widehat{\mathcal{I}}_n(\theta)$ to be  positive definite.
\end{assumption}

Let $\mathcal{F}_N = \sigma\left\{(Y_{n+1},X_{n+1}),\ldots,(Y_N,X_N)\right\}$, and consider $p = 1$ and $\theta \in \mathbb{R}$ for now. Then we have that $(\theta_N, \mathcal{F}_N)_{N \geq n}$ is a martingale as
\begin{align*}
    E\left\{s(\theta_{N-1}, Y_{N}; X_{N}) \mid \mathcal{F}_{N-1}\right\} = E\left[E\left\{s(\theta_{N-1}, Y_{N}; X_{N}) \mid \mathcal{F}_{N-1}, X_N\right\} \mid \mathcal{F}_{N-1}\right]
\end{align*}
where the inner conditional expectation is 0 almost surely, regardless of the distribution of $X_N$. Furthermore, we have that the variance of the score function is as expected:
\begin{align*}
    E\left\{\widehat{\mathcal{I}}_n(\theta_{N-1})^{-2}s(\theta_{N-1}, Y_{N}; X_{N})^2 \mid \mathcal{F}_{N-1}\right\} &= \widehat{\mathcal{I}}_n(\theta_{N-1})^{-2}E\left[E\left\{s(\theta_{N-1}, Y_{N}; X_{N})^2 \mid \mathcal{F}_{N-1}, X_N\right\} \mid \mathcal{F}_{N-1}\right]\\
    &= \widehat{\mathcal{I}}_n(\theta_{N-1})^{-2} E\left[\mathcal{I}(\theta_{N-1}; X_N)\mid \mathcal{F}_{N-1}\right] = \widehat{\mathcal{I}}_n(\theta_{N-1})^{-1}
\end{align*}
where the final equality follows as $X_N \sim p_n^x$. This clearly extends for the case where $p > 1$.

We can now state the regression versions of Assumption \ref{as:UI_mv} and Proposition \ref{prop:Z4_mv}, where  $\widehat{\mathcal{I}}_n(\theta)$ is considered as fixed.   
 \begin{assumption}\label{as:UI_reg}
For each $t \in \R^p$, let $Z_N(t)$ denote the natural gradient of the Cram\'{e}r-Wold device:
\begin{align*}
    Z_N(t) =  t^T  \widehat{\mathcal{I}}_n(\theta_{N-1})^{-1}s(\theta_{N-1}, Y_N;X_N).
\end{align*}
The sequence $\left(Z^2_N(t)\right)_{N >n}$ is uniformly integrable under predictive resampling starting from the initial $\theta_n$.
\end{assumption}
\begin{corollary}\label{corr:Z4_reg}
       Let $Z(\theta,Y; X) = \widehat{\mathcal{I}}_n(\theta)^{-1}s(\theta, Y; X)$, where  $X \sim p_n^x$ and $Y \mid X \sim p_\theta(\cdot \mid X)$. Suppose there exists non-negative constants $B,C< \infty$ such that the following holds for all $\theta \in \Theta \subseteq \R^p$:
    \begin{align*}
    E_{\theta,X}\left\{\|Z(\theta,Y;X)\|^4\right\} \leq B + C \|\theta\|^4.
    \end{align*}
   Then $\left(Z^2_N(t)\right)_{N > n}$ is uniformly integrable for each $t \in \R^p$. 
\end{corollary}
The proof is practically identical to Proposition \ref{prop:Z4_mv}, so we omit it.

\subsection{Predictive central limit theorem}
For the predictive central limit theorem, the proof of the regression case is identical to the multivariate case (Theorem \ref{th:pred_CLT_mv}), so we omit the proof. We define the covariate matrix estimate as
\begin{align}\label{eq:pred_as_cov_reg}
    \widehat{\Sigma}_N = r_N^2\widehat{\mathcal{I}}_n(\theta_N)^{-1},
    \end{align}
    where once again we have $r_N^2 =\sum_{i = N+1}^\infty i^{-2}$ and $\widehat{\mathcal{I}}_n(\cdot)$ is considered fixed.
\begin{theorem}\label{th:pred_CLT_reg}
    Under Assumptions \ref{as:existence_reg} and \ref{as:UI_reg}, the parametric  martingale posterior arising from (\ref{eq:sgd_reg}) satisfies
    \begin{align*}        \widehat{\Sigma}_N^{-\frac{1}{2}}\left(\theta_\infty - \theta_N\right) \to \mathcal{N}(0,I)
    \end{align*}
    $\mathcal{F}_\infty$-mixing, where $\widehat{\Sigma}_N$ is defined in (\ref{eq:pred_as_cov_reg}).
\end{theorem}

\subsection{Bernstein-von Mises theorem}
Once again, we begin by letting $(\theta_n)_{n\geq 1}$ be a deterministic sequence converging to $\theta^*$. The Bernstein-von Mises theorem requires a bit more care in the regression case, as the gradient preconditioner depends on the fixed design sequence $X_{1:n}$. 
We will now take $n\to \infty$, so $\widehat{\mathcal{I}}_n(\theta)$ will vary, unlike in previous subsections. We will consider the same predictive resampling scheme for the covariates from $p_n^x$ as in the previous subsections.
As in \citet[Example 2.28]{van2000}, we can make assumptions on the deterministic sequence $X_{1:\infty}$ in order to apply Theorem \ref{th:bvm_as_mv}. 
We now consider the relevant assumptions.
\begin{assumption}\label{app:as:bvm_reg}
        Suppose there exists $m$ such that Assumption \ref{as:existence_reg} holds for all $n \geq m$, and the limit $\mathcal{I}(\theta^*) = \lim_{n \to \infty} \widehat{\mathcal{I}}_n(\theta_n)$ exists. 
    \end{assumption}
The requirement of $n \geq m$ arises in the regression setting as $\widehat{\mathcal{I}}_n(\theta)$ may not be invertible if $n < p$, for example if $\widehat{\mathcal{I}}_n(\theta)$ contains a term like $n^{-1}\sum_{i = 1}^n X_i X_i^T$.
Note that the above is a deterministic convergence, as we are in the fixed design setting.
    
    We now introduce a regression version of Proposition \ref{prop:Z4_mv}, where we let
\begin{align*}
     \widetilde{X}_{ni} \sim p_n^x, \, \,  \widetilde{Y}_{ni}  \sim p_{\theta_{n,i-1}}(\cdot \mid \widetilde{X}_{ni}),\quad 
    \theta_{ni} = \theta_{n,i-1} + (n+i)^{-1}Z_{ni}
\end{align*}
and
\begin{align*}
    Z_{ni}  = \widehat{\mathcal{I}}_n\left(\theta_{n,i-1}\right)^{-1}s\left(\theta_{n,i-1}, \widetilde{Y}_{ni}; \widetilde{X}_{ni}\right)
\end{align*}
for  $i \geq 1, n \geq m$. Again, the condition to be checked will be an implicit assumption on the design matrix. 
\begin{proposition}\label{prop:Z4_reg}
       Let $Z_n(\theta,Y; X) = \widehat{\mathcal{I}}_n(\theta)^{-1}s(\theta, Y; X)$, where  $X \sim p_n^x\, $ and $\, Y \mid X \sim p_\theta(\cdot \mid X)$. Suppose there exists non-negative constants $B,C< \infty$ independent of $n$ such that the following holds for all $\theta \in \Theta \subseteq \R$:
    \begin{align*}
    E_{\theta,X}\left\{\|Z_n(\theta,Y;X)\|^4\right\} \leq B + C \|\theta\|^4
    \end{align*}
     for all $n\geq m$. Furthermore, let $(\theta_n)_{n \geq 1}$ be a deterministic convergent sequence. Then $\left(Z^2_{ni}(t)\right)_{i\geq 1, n \geq m}$ is uniformly integrable for each $t \in \mathbb{R}^p$, where $Z_{ni}(t) = t^T Z_{ni}$. 
\end{proposition}
\begin{proof}
 The proof starts with an identical argument to that of Proposition \ref{prop:Z4_mv}. We can just need to check the following term is 0:
 \begin{align*}
    E\left(\langle \theta_{n,i-1}, Z_{ni}\rangle \mid \mathcal{F}_{n,i-1}\right) =  \langle \theta_{n,i-1}, E\left( Z_n \mid \mathcal{F}_{n,i-1}\right)\rangle = 0.
 \end{align*}
 This holds from the earlier argument that $E(Z_{ni} \mid \mathcal{F}_{n,i-1}) = E\left\{E\left(Z_{ni} \mid X_{ni},\mathcal{F}_{n,i-1}\right)\mid \mathcal{F}_{n,i-1}\right\} =0_p$.

 From the same argument as Proposition \ref{app:prop:Z4_bvm}, we can show that
 \begin{align*}
     E\left(\left\|Z_{ni}\right\|^4\right)&\leq D_1\|\theta_{n0}\|^4 + D_2 \sum_{i = n+1}^\infty i^{-2}  \\
     &\leq D_1\|\theta_n\|^4 + D_2n^{-1}
 \end{align*}
 for all $i \geq 1, n\geq m$,
 where $D_1,D_2 <\infty$ are non-negative finite constants which can be chosen independently of $n$, and we have used $\theta_{n0} = \theta_n$. As $\theta_n$ is convergent, we then have
\begin{align*}
    \sup_{n\geq m}\sup_{i \geq 1} E\left(\left\|Z_{ni}\right\|^4\right) < \infty,
\end{align*}
which implies $\left( Z^2_{ni}(t): i \geq 1, n\geq m \right)$ is uniformly integrable for each $t \in \mathbb{R}^p$ following the same argument as in Proposition \ref{prop:Z4_mv}.
\end{proof}

Finally, we require an extension of the continuous differentiability condition from Assumption \ref{as:bvm}, which is new for the regression setting due to the estimation of $\widehat{\mathcal{I}}_n(\cdot)$.
\begin{assumption}\label{app:as:bvm_reg2} 
      Suppose there exists a neighbourhood $U$ of $\theta^*$ such that for all $\theta \in U$, we have that $\mathcal{I}(\theta)^{-1}$ is continuously differentiable
      and
      \begin{align*}
          \sup_{\theta \in U} \, \rho\left\{\widehat{\mathcal{I}}_n(\theta)^{-1} - \mathcal{I}(\theta)^{-1}\right\} \to 0
      \end{align*}
      as $n \to \infty$, where $\rho(A)$ is the spectral radius of the matrix $A$. Furthermore, assume the conditions of Proposition \ref{prop:Z4_reg} hold.
\end{assumption}

We are now ready to state the main result for the regression setting. 
\begin{theorem}\label{th:bvm_as_reg}
     Let $(\theta_n)_{n \geq 1}$ be a deterministic sequence which converges to $\theta^*$ as $n \to \infty$. Under Assumptions \ref{app:as:bvm_reg} and \ref{app:as:bvm_reg2}, we have
    \begin{align*}
       r_n^{-1}\left(\theta_{n\infty} - \theta_n\right) \to \mathcal{N}\left\{0,\mathcal{I}(\theta^*)^{-1}\right\}
    \end{align*}
    in distribution
    $P^*$-almost surely as $n \to \infty$, where $r_n^2 = {\sum_{i = n+1}^\infty i^{-2}}$.
\end{theorem}
\begin{proof}
    The only deviation to the non-regression case will be when verifying  the variance condition (\ref{app:infin_array_as5}), specifically Lemma \ref{app:lem:var_bvm}. Consider the univariate case for now. For some $r \in \mathbb{R}^+$, consider a compact interval $ I = [\theta^* - r, \theta^*+r]$ where $I \subset U$. 
Now consider the term
\begin{align*}
    \sup_{i \geq 1} E\left\{\left| \widehat{\mathcal{I}}_n(\theta_{ni})^{-1} - \mathcal{I}(\theta^*)^{-1}\right|\right\} &\leq  \sup_{i \geq 1} E\left\{\left| \widehat{\mathcal{I}}_n(\theta_{ni})^{-1} - \mathcal{I}(\theta^*)^{-1}\right|\mathbbm{1}\left(\theta_{ni} \in I\right)\right\}\\
    &+ \sup_{i \geq 1} E\left\{\left| \widehat{\mathcal{I}}_n(\theta_{ni})^{-1} - \mathcal{I}(\theta^*)^{-1}\right|\mathbbm{1}\left(\theta_{ni} \in I^C\right)\right\},
\end{align*}
where we have split the expectation into $\theta_{ni}$ lying inside or outside $I$. 

The second term in the above can be handled as before. Since $\left(Z_{ni}^2\right)_{i \geq 1, n \geq m}$ is in $L_1$, and $E\left[Z_{ni}^2 \mid \mathcal{F}_{n,i-1} \right] = \widehat{\mathcal{I}}_n(\theta_{n,i-1})^{-1}$, we have that $(\widehat{\mathcal{I}}_n(\theta_{n,i-1})^{-1})_{i \geq 1, n \geq m}$ is uniformly integrable.\vspace{2mm}

For the first term, the condition that we want is 
\begin{align*}
     \sup_{i \geq 1} E\left\{\left| \widehat{\mathcal{I}}_n(\theta_{ni})^{-1} - \mathcal{I}(\theta^*)^{-1}\right|\mathbbm{1}\left(\theta_{ni} \in I\right)\right\} \leq K \sup_{i \geq 1} E\left(|\theta_{ni}- \theta^*|\right)\to 0. 
\end{align*}
A sufficient condition for the above is that for $\theta_{ni}\in I$, the neighborhood $I$ satisfies
\begin{align*}
    \left| \widehat{\mathcal{I}}_n(\theta_{ni})^{-1} - \mathcal{I}(\theta^*)^{-1}\right| \leq K \left|\theta_{ni} - \theta^* \right|
\end{align*}
for some finite constant $K$  for all $n \geq m$. To get a handle on this, the triangle inequality gives
\begin{align*}
    \left| \widehat{\mathcal{I}}_n(\theta_{ni})^{-1} - \mathcal{I}(\theta^*)^{-1}\right| \leq  \left| \widehat{\mathcal{I}}_n(\theta_{ni})^{-1} - {\mathcal{I}}(\theta_{ni})^{-1}\right| +  \left| {\mathcal{I}}(\theta_{ni})^{-1} - \mathcal{I}(\theta^*)^{-1}\right|.
\end{align*}
The second term can be controlled by the continuous differentiability of $\mathcal{I}(\theta)^{-1}$ within $I$ as before. The first term then satisfies
\begin{align*}
    \sup_{i \geq 1}\left| \widehat{\mathcal{I}}_n(\theta_{ni})^{-1} - {\mathcal{I}}(\theta_{ni})^{-1}\right| \to 0
\end{align*}
from the final part of Assumption \ref{app:as:bvm_reg2}.
This then gives
\begin{align*}
     \sup_{i \geq 1} E\left\{\left| \widehat{\mathcal{I}}_n(\theta_{ni})^{-1} - \mathcal{I}(\theta^*)^{-1}\right|\mathbbm{1}\left(\theta_{ni} \in I\right)\right\} \to 0. 
\end{align*}
Finally, this gives
\begin{align*}
     \sup_{i \geq 1} E\left\{\left| \widehat{\mathcal{I}}_n(\theta_{ni})^{-1} - \mathcal{I}(\theta^*)^{-1}\right|\right\} \to 0
\end{align*}
as $n \to \infty$ as desired. The remainder of the proof follows as in Theorem \ref{th:bvm}, where we can show a regression version of Lemma \ref{app:lem:var_bvm}. The Lindeberg condition then follows from the uniform integrability of $\left(Z_{ni}^2\right)_{i \geq 1, n \geq m}$.\vspace{1mm}

The multivariate case is a simple extension of the above. For each $t \in \R$:
\begin{align*}
  \sup_{i \geq 1}  E\left[\left|t^T \left\{\widehat{\mathcal{I}}_n(\theta_{ni})^{-1} -\mathcal{I}(\theta^*)^{-1} \right\}t\right|\right] &\leq   \sup_{i \geq 1}  E\left[\left|t^T \left\{\widehat{\mathcal{I}}_n(\theta_{ni})^{-1} -\mathcal{I}(\theta^*)^{-1} \right\}t\right|\mathbbm{1}\left(\theta_{ni} \in I\right)\right]\\
  &+   \sup_{i \geq 1}  E\left[\left|t^T \left\{\widehat{\mathcal{I}}_n(\theta_{ni})^{-1} -\mathcal{I}(\theta^*)^{-1} \right\}t\right|\mathbbm{1}\left(\theta_{ni} \in I^C\right)\right].
\end{align*} 
The second term is as before, controlled by $\left(Z^2_{ni}(t)\right)_{i \geq 1, n \geq 1}$ being uniformly integrable.\vspace{1mm}

For the first term, we consider $\theta_{ni} \in I$ and
\begin{align*}
    \left|t^T\left\{ \widehat{\mathcal{I}}_n(\theta_{ni})^{-1} - \mathcal{I}(\theta^*)^{-1} \right\} t\right| \leq  \left|t^T\left\{ \widehat{\mathcal{I}}_n(\theta_{ni})^{-1} - {\mathcal{I}}(\theta_{ni})^{-1}\right\}t\right| +  \left|t^T\left\{ {\mathcal{I}}(\theta_{ni})^{-1} - \mathcal{I}(\theta^*)^{-1}\right\}t\right|.
\end{align*}
The second term is controlled by the continuous differentiability of $t^T\mathcal{I}(\theta)^{-1}t$. We then use Assumption \ref{app:as:bvm_reg2} for the remaining first term:
\begin{align*}
    \sup_{i \geq 1}\left|t^T\left\{ \widehat{\mathcal{I}}_n(\theta_{ni})^{-1} - {\mathcal{I}}(\theta_{ni})^{-1}\right\}t\right|  \leq \|t\|^2\sup_{i \geq 1}\,  \rho\left\{\widehat{\mathcal{I}}_n(\theta_{ni})^{-1} - {\mathcal{I}}(\theta_{ni})^{-1}\right\} \to 0,
\end{align*}
where $\rho(A)$ is the spectral radius of $A$.
The above follows from a Rayleigh quotient argument as the difference of symmetric matrices is still symmetric (although not necessarily positive definite). Again, the rest of the proof is identical to that of Theorem \ref{th:bvm_as_mv} based on the uniform integrability of $\left(Z_{ni}^2(t)\right)_{i \geq 1, n \geq m}$.
\end{proof}

\setcounter{equation}{0}
\setcounter{theorem}{0}
\setcounter{proposition}{0}
\setcounter{lemma}{0}
\setcounter{corollary}{0}
\setcounter{assumption}{0}
\section{Examples}
\subsection{Introduction}
In this section, we illustrate the verification of the assumptions for a few multivariate and regression examples, complementing the univariate examples in the main paper. In general for the multivariate case, Assumption \ref{as:existence_mv} and the first part of Assumption \ref{as:bvm_mv} is easy to check, as it only depends on the Fisher information matrix which usually has a continuously differentiable inverse. The crux is often checking the uniform integrability condition (Assumption \ref{as:UI_mv} and the second part of Assumption \ref{as:bvm_mv}) using the fourth moment bound in  Proposition \ref{prop:Z4_mv}. 

For the regression examples, again Assumption \ref{as:existence_reg} is usually straightforward, with the conditions of Proposition \ref{prop:Z4_reg} being the crux. We will also have to put in more effort into checking Assumptions \ref{app:as:bvm_reg} and the first part of Assumption \ref{app:as:bvm_reg2} as they depend on the design matrix.
 
\subsection{Exponential example}\label{sec:expon}
Here, we verify the conditions for Proposition \ref{prop:Z4} for the exponential distribution example from the main paper. A standard calculation gives
\begin{align*}
    Z_N &= \theta_{N-1}^{2}\left(-\frac{1}{\theta_{N-1}} + \frac{1}{\theta_{N-1}^{2} }Y_N\right) = Y_N - \theta_{N-1}.
\end{align*}
We can then compute the 4th centralized moment of $Y_N$, which gives
\begin{align*}
    E\left(Z_N^4 \mid \mathcal{F}_{N-1}\right)  = 9\theta_{N-1}^4.
\end{align*}
We thus have the condition with $B = 0$.

\subsection{Normal unknown mean and variance example}\label{sec:normal_mv}
We illustrate the checking of the conditions for Proposition \ref{prop:Z4_mv} for the unknown mean and variance normal model, where $\theta = [\mu,\sigma^2]^T$. A standard calculation gives
\begin{align*}
    s(\theta,Y) &= \begin{bmatrix}
    \dfrac{Y - \mu}{\sigma^2} \vspace{1mm}\\
    -\dfrac{1}{2\sigma^2}+ \dfrac{(Y- \mu)^2}{2\sigma^4} \vspace{1mm}
    \end{bmatrix}, \quad 
\mathcal{I}(\theta) = \begin{bmatrix}
   \dfrac{1}{\sigma^2}\vspace{1mm} & \hspace{1mm}0\\
   0 &\dfrac{1}{2\sigma^4}\vspace{1mm}
\end{bmatrix},
\end{align*}
which then gives
\begin{align*}
    Z(\theta,Y) = \begin{bmatrix}
    {Y - \mu}\\
   (Y- \mu)^2- \sigma^2
    \end{bmatrix}\cdot
\end{align*}
The norm can be computed as
\begin{align*}
    \|  Z(\theta,Y)\|^4 &= \left[(Y - \mu)^2 + \left\{(Y- \mu)^2- \sigma^2\right\}^2\right]^2.
\end{align*}
Once again, we have 
\begin{align*}
    Q = \frac{(Y -\mu)^2}{\sigma^2}\sim \chi^2(1),
\end{align*}
which gives
\begin{align*}
    E_\theta\left( \|  Z(\theta,Y)\|^4\right)&= \sigma^4 E\left[\left\{Q + \sigma^2(Q-1)^2\right\}^2\right]=\sigma^4 \left(3 + 60 \sigma^4 + 20\sigma^2\right) \leq B + C(\sigma^2)^4 
\end{align*}
We have
\begin{align*}
    \|\theta\|^4 = \left\{\mu^2 + (\sigma^2)^2\right\} ^2 \geq (\sigma^2)^4.
\end{align*}
As a result, we have
\begin{align*}
      E_\theta\left( \|  Z(\theta,Y)\|^4\right)&\leq B + C \|\theta\|^4
\end{align*}
as desired for Proposition \ref{prop:Z4_mv}.\\
\subsection{Multivariate normal example}\label{sec:app_mvnormal}
The condition for Proposition \ref{prop:Z4_mv} can be similarly checked for the multivariate normal case as in Section \ref{sec:sim} of the main paper, although notation becomes much more involved. Consider the bivariate Gaussian case, where $\theta = [\mu_1, \mu_2, s_1, s_1, s_{12}]$ 
with the predictive as $\mathcal{N}(\mu, \Sigma)$
with $\mu = [\mu_1,\mu_2]^T$ and 
\begin{align*}
    \Sigma= \begin{bmatrix}
    s_1 & s_{12}\\
    s_{12} & s_2
\end{bmatrix}.
\end{align*}
This then gives
\begin{align}\label{eq:mv_normal_update}
    Z(\theta,Y) = \begin{bmatrix}
        Y_1- \mu_1\\
        Y_2- \mu_2\\
        (Y_1 -\mu_1)^2- s_1\\
        (Y_2 -\mu_2)^2- s_2\\
       (Y_1- \mu_1)(Y_2 - \mu_2) - s_{12}\\
    \end{bmatrix}, \quad
    \mathcal{I}(\theta)^{-1} = \begin{bmatrix}
        \Sigma & 0 \\
        0 & J
    \end{bmatrix}
\end{align}
where
\begin{align*}
    J = 2\begin{bmatrix}
        s_1^2 &s_{12}^2 & s_1 s_{12}\vspace{1mm}\\
        s_{12}^2 & s_2^2 & s_2 s_{12}\vspace{1mm}\\
        s_1 s_{12} \hspace{1mm}& s_2 s_{12}\hspace{1mm} & \frac{1}{2}{(s_{12}^2 + s_1 s_2)}
        \end{bmatrix}
\end{align*}
Although tedious, one can show that $E[\|Z(\theta,Y)\|^4]$ is a multivariate polynomial in $(s_1,s_2,s_{12})$ with degree 4. As a result, we have
\begin{align*}
    E[\|Z(\theta,Y)\|^4] \leq B_1 + C_1 s_1^4 + C_2 s_2^4 + C_3 s_{12}^4 \leq B + C \|\theta\|^4
\end{align*}
for some finite positive constants $B,C,B_1,C_1,C_2,C_3$, thereby satisfying the conditions for Proposition \ref{prop:Z4_mv}.

For the general case with $p>2$ dimensions, the update function has the same form, with
\begin{align*}
    Z_{\mu_j}(\theta,Y) = Y_j - \mu_j,\quad Z_{s_{jk}}(\theta,Y) = (Y_j - \mu_j)(Y_k - \mu_k) - s_{jk}.
\end{align*}
 for $j,k \in \{1,\ldots,p\}$, where $s_{jj} = s_j$ and $s_{kl} = s_{lk}$. The update vector is of length $ M+p $, where $M= p(p+1)/2$ is the number of unique entries in $\Sigma$.
 The Fisher information is the same form as in (\ref{eq:mv_normal_update}), though notation is a bit trickier now. 
The matrix $J$ is now of size $M \times M$.  
For entries corresponding to $(s_{jk}$, $s_{lm})$ with $j,k,l,m \in \{1,\ldots,p\}$ we have
$$J_{s_{jk},s_{lm}}= {s_{jm}s_{lk} + s_{jl} s_{mk}} \cdot$$
Again, we can extend the above proof to show that $E\left[\|Z(\theta,Y\|^4\right]$ is a degree 4 polynomial in $(s_{jk})_{1 \leq j,k \leq p}$ so again Proposition \ref{prop:Z4_mv} applies.

\subsection{Normal linear regression}\label{sec:norm_reg}
    We now verify Assumptions \ref{app:as:bvm_reg} and \ref{app:as:bvm_reg2} for  normal linear regression where for $y \in \mathbb{R}$ and $x \in \mathbb{R}^p$, our predictive is $p_\theta(y \mid x) = \mathcal{N}(y \mid \beta^T x, \sigma^2)$ with $\theta = [\beta, \sigma^2]^T$. The estimated Fisher information is
    \begin{align*}
    \widehat{\mathcal{I}}_n(\theta) = 
    \begin{bmatrix}\dfrac{\Sigma_{n,x}}{\sigma^2} & 0\\
    0 &\dfrac{1}{\sigma^4}
    \end{bmatrix}
    \end{align*}
where $\Sigma_{n,x} = n^{-1}\sum_{i = 1}^n X_i X_i^T$. Let $(\theta_n)_{n \geq 1}$ be a deterministic sequence converging to $\theta^*$, where ${\sigma^*}^2 > 0$. We omit the score function as it is a standard result. We now check Assumptions \ref{app:as:bvm_reg} and \ref{app:as:bvm_reg2}, which is enough for both the predictive and frequentist central limit theorems to hold.

We now make two reasonable assumptions on the sequence of design points.
\begin{assumption}\label{as:PSD_reg}
 The limit $\Sigma_x = \lim_{n\to \infty} \Sigma_{n,x}$ exists and is positive definite.
\end{assumption}
\begin{assumption}\label{as:bounded_reg}
The design values are bounded, that is $\sup_{n \geq 1}\|X_n\|^2 = K_{\text{max}} < \infty$.
\end{assumption}

 For Assumption \ref{app:as:bvm_reg}, we clearly have that $\widehat{\mathcal{I}}_n(\theta)$ is continuous in $\sigma^2$ for all $n$.  Furthermore, it is clear that  $\widehat{\mathcal{I}}_n(\theta)$ is positive definite for all sufficiently large $n$ from Assumption \ref{as:PSD_reg} for each value of $\sigma^2$. Finally, $\Sigma_{n,x}/{\sigma_n^*}^2$ converges as $\Sigma_{n,x}$ and $\sigma^2_n$ both converge, so the limit $\mathcal{I}(\theta^*) = \lim_{n\to \infty} \widehat{\mathcal{I}}_n(\theta_n)$ exists.  

 For Assumption \ref{app:as:bvm_reg2}, we first check the condition of Proposition \ref{prop:Z4_reg}. This is just an extension to Section \ref{sec:normal_mv}. One can verify that 
 \begin{align*}
     Z_n(\theta, Y; X) =\begin{bmatrix}
        (Y- \beta_n^T X)\, \Sigma_{n,x}^{-1}X \\
        \left(Y- \beta_n^T X\right)^2 - \sigma_n^2
     \end{bmatrix}
 \end{align*}
which gives
\begin{align*}
    \|Z_n(\theta,Y;X)\|^4 = \left\{(Y- \beta_n^T X)^2\, X^T\Sigma_{n,x}^{-2}X   + \left\{     \left(Y- \beta_n^T X\right)^2 - \sigma_n^2\right\}^2 \right\}^2
\end{align*}
We first upper bound $X^T\Sigma_{n,x}^{-2}X$ using Assumptions \ref{as:PSD_reg} and \ref{as:bounded_reg}. As the limit $\Sigma_x$ is positive definite, there exists $\delta > 0$ such that $\lambda_{\text{min}}(\Sigma_{n,x}) > \delta$ eventually, where $\lambda_{\text{min}}(A)$ is the minimum eigenvalue of the matrix $A$. This then gives
\begin{align*}
    X^T\Sigma_{n,x}^{-2}X \leq \frac{K_{\text{max}}}{\delta^2}
\end{align*}
eventually, so there exists a finite positive constant $L$ such that
\begin{align*}
    \|Z_n(\theta,Y;X)\|^4 = \left\{ L \varepsilon^2   + \left(     \varepsilon^2 - \sigma^2\right)^2 \right\}^2
\end{align*}
where $\varepsilon = (Y-\beta^T X) \sim \mathcal{N}(0,\sigma_n^2)$ under predictive resampling. We can thus carry out the exact same argument as Section \ref{sec:normal_mv}, which gives
\begin{align*}
    E\left[\|Z_n(\theta,Y;X)\|^4\right] \leq B + C\left(\sigma_n^2\right)^4 \leq B + C \|\theta\|^4
\end{align*}
for some finite positive constants $B, C$, where $L,B,C$ can all be chosen independently of $n$.

Finally, we want to check the uniform convergence condition, which is relatively straightforward. Pick a neighborhood $U = (\theta^* - r, \theta^* + r)$ for some small positive $r < {\sigma^*}^2$. We then have
\begin{align*}
    \sup_{\theta \in U} \|\widehat{\mathcal{I}}_n(\theta)^{-1} - \mathcal{I}(\theta)^{-1}\| \leq \left({\sigma^*}^2 + r\right){\rho\left(\Sigma_{n,x}^{-1} - \Sigma_{x}^{-1} \right)}\cdot 
\end{align*}
By Assumption \ref{as:PSD_reg}, we have that $\rho\left(\Sigma_{n,x}^{-1} - \Sigma_{x}^{-1} \right) \to 0$ from the continuity of the spectral radius, so $U$ satisfies Assumption \ref{app:as:bvm_reg2}. 

\subsection{Robust linear regression}\label{sec:app_robreg}
Consider the robust linear regression case as in Section \ref{sec:rwd}, where the model is $Y = \beta^T X + \tau \varepsilon$, where $\varepsilon$ is a standard Student's t-distribution with $\nu > 1$ degrees of freedom.
Our parameter of interest is then $\theta = [\beta,\tau^2]^T$, 
 where this is an extension of Example \ref{ex:t} and Section \ref{sec:norm_reg}. We compute the score function and Fisher information matrix:
\begin{align*}
\nabla_{\beta}{ \log p_\theta(Y \mid X)} = \frac{( \nu +1)( Y-\beta^T X) X}{\nu \tau^2 + (Y-\beta^T X )^2}, &\quad \frac{\partial \log p_\theta(Y \mid X)}{\partial \tau^2} = 
\frac{\nu\{(Y-\beta^T X)^2 - \tau^2\}}{2\tau^2 \left\{\nu \tau^2 + (Y- \beta^T X)^2\right\}}
\end{align*}
\begin{align*}
\widehat{\mathcal{I}}_n(\theta) &= \begin{bmatrix}
    \dfrac{(\nu + 1)\Sigma_{n,x}}{(\nu + 3)\tau^2} & 0\\
    0 & \dfrac{\nu}{2(\nu + 3)\tau^4}
\end{bmatrix}
\end{align*}
where $\Sigma_{n,x} = n^{-1}\sum_{i = 1}^n X_i X_i^T$ as before. 
This gives us  $Z_n(\theta,Y; X) = [Z_{\beta}(\theta,Y;X),Z_{\tau^2}(\theta,Y;X) ]^T$  with
\begin{align*}
Z_{\beta}(\theta,Y;X) &= \left\{ \frac{\tau^2(\nu + 3)(Y- \beta^T X)}{\nu \tau^2 + (Y-\beta^T X)^2}\right\}\Sigma_{n,x}^{-1}X,\\
Z_{\tau^2}(\theta,Y;X) &=  \frac{\tau^2(\nu + 3)\left\{\left(Y - \beta^T X\right)^2 - \tau^2\right\}}{\nu \tau^2 + \left(Y-\beta^T X\right)^2}
\end{align*}
as in the main paper. Again, we make Assumptions \ref{as:PSD_reg} and \ref{as:bounded_reg} and check Assumptions \ref{app:as:bvm_reg} and \ref{app:as:bvm_reg2}. Verifying Assumption \ref{app:as:bvm_reg}  is exactly the same as Section \ref{sec:norm_reg} so we omit this.
For Assumption \ref{app:as:bvm_reg2}, we first check the conditions for Proposition \ref{prop:Z4_reg}, where we want to upper bound
$E\left\{\|Z_n(\theta,Y;X)\|^4\right\}$. Conditional on $X$, we have that $R = (Y - \beta^T X)/ \tau \sim t(\nu)$, so we can write the above as
\begin{align*}
Z_{\beta}(\theta,Y;X) &= \left\{ \frac{\tau(\nu + 3)R}{ \nu  +R^2}\right\}\Sigma_{n,x}^{-1}X,\quad 
Z_{\tau^2}(\theta,Y;X) =  \frac{\tau^2(\nu + 3)\left(R^2 - 1\right)}{\nu  +R^2}
\end{align*}

We will have the following useful properties for $\nu > 1$:
\begin{align*}
   \left( \frac{R}{\nu + R^2}\right)^2 \leq \frac{1}{4\nu}, \quad 
   \left( \frac{R^2 -1}{\nu + R^2}\right)^2 \leq 1.
\end{align*}
The above bounds are simpler than the normal case. This then gives us
\begin{align*}
    \|Z_n(\theta,Y;X)\|^4 & = \left\{\|Z_{\beta}(\theta, Y;X)\|^2 + Z_{\tau^2}(\theta,Y;X)^2\right\}^2\\
    &\leq  (\nu + 3)^4\tau^4\left\{\frac{\|\Sigma_{n,x}^{-1}X\|^2}{4\nu} + \tau^2 \right\}^2\\
    &=  (\nu + 3)^4 \tau^4\left\{\frac{\|\Sigma_{n,x}^{-1}X\|^4}{16\nu^2} + \tau^4 + \frac{\tau^2\|\Sigma_{n,x}^{-1}X\|^2}{2\nu} \right\}.
\end{align*}
Once again, Assumptions \ref{as:PSD_reg} and \ref{as:bounded_reg} give us
\begin{align*}
\|\Sigma_{n,x}^{-1}X \|^2 = X^T \Sigma_{n,x}^{-T} \Sigma_{n,x}^{-1} X \leq \frac{\|X\|^2} {\lambda_{\text{min}}(\Sigma_{n,x})^2} \leq \frac{K_{\text{max}}}{\delta^2},
\end{align*}
for sufficiently large $n$. Since the above constants are chosen independently of $n$, we have that 
there exists positive $B,C < \infty$ such that
\begin{align*}
E\left\{ \|Z_n(\theta,Y;X)\|^4\right\} 
&\leq B + C (\tau^2)^4\leq B + C \|\theta\|^4
\end{align*}
for all $n$. For the uniform convergence condition, the argument is exactly the same as Section \ref{sec:norm_reg} with some extra scaling constants that depend on $\nu$.

\subsection{Logistic regression}\label{sec:app_logreg}

    For linear logistic regression, a standard calculation gives
    \begin{align*}
        s(\theta,Y; X) = \{Y - \sigma(\theta^{T}X )\}X, \quad \mathcal{I}(\theta; X) = \sigma(\theta^T X)\{1 - \sigma(\theta^T X)\}XX^T
    \end{align*}
    where $\sigma(\cdot)$ is the sigmoid function, and  we assume $X$ contains an intercept term as well. However, the Hessian term is challenging to bound as it grows exponentially with $\theta^T X$, so we instead define
    \begin{align*}
        {\mathcal{I}_{\kappa}}(\theta;X) = \max\{\sigma(\theta^T X)\{1 - \sigma(\theta^T X)\}, \kappa\} X X^T
    \end{align*}
    for some $\kappa > 0$. We then define
    \begin{align*}
        \widehat{\mathcal{I}}_{n,\kappa}(\theta) = \frac{1}{n}\sum_{i = 1}^n \mathcal{I}_{\kappa}(\theta;X_i).
    \end{align*}
We then have the update term as $Z(\theta,Y;X) = \widehat{\mathcal{I}}_{n,\kappa}(\theta)^{-1} s(\theta, Y;X)$.
This is still a martingale, as we are free to pick the preconditioner as long as it is $\mathcal{F}_{N}$-measurable. 
Again, we assume Assumptions \ref{as:PSD_reg} and \ref{as:bounded_reg}.

Noting that $|Y- \sigma(\theta^T X)|\leq 1$, we have
\begin{align*}
      \|\widehat{\mathcal{I}}_{n,\kappa}(\theta)^{-1} s(\theta, Y;X)\|^4 \leq \frac{\|s(\theta,Y;X)\|^4}{\kappa^4 \lambda_{\text{min}} \left(\Sigma_{n,x}\right)^4}
      \leq \frac{K_{\text{max}}^4}{\kappa^4 \delta ^4}< \infty
\end{align*}
for all $n \geq m$. This is finite with all constants independent of $n$, so we can automatically apply Proposition \ref{prop:Z4_reg} with $C = 0$. We suspect that Theorem \ref{th:bvm_as_reg} can be extended to this case as well, provided the truncation by $\kappa$ does not happen as $n$ increases.

\setcounter{table}{0}
\setcounter{figure}{0}

\section{Additional Experiments}
In this section, we provide additional experimental results for the simulation and real data analysis.

\subsection{Simulation}

In this section, we introduce additional results for the simulation from the main paper. Table \ref{tab2_supp} shows the additional results for $(\mu_2,s_2)$ for the Bayesian and martingale posterior, which match that of Table \ref{tab2} in the main paper as expected.

\begin{table}
    \caption{Coverage and average length of 95\% credible intervals of traditional Bayesian posterior (Markov Chain Monte Carlo) and the parametric martingale posterior (hybrid predictive resampling) on  different sample sizes \vspace{2mm}}
    \small
    \center
{   \begin{tabular}{ccrrcrrcrr}
        && \multicolumn{2}{c}{$n=20$} & &\multicolumn{2}{c}{$n=100$}&&\multicolumn{2}{c}{$n=500$}\vspace{2mm} \\
        Method& Parameter & Cov & Length&&Cov & Length && Cov & Length \\ 
        \multirow{5}{*}{Bayesian posterior}&$\mu_2$ & 94.2& 6.5&& 92.8 & 2.8&& 94.2&1.2\\ 
& $s_2$& 95.2 & 7.7 && 95.4 &2.9&& 93.4 & 1.3\\[5pt]
   &{Max. SE} & 1.0& 0.1 && 1.1 & 0.02 && 1.1 &0.004\\[2pt]
   & Run-time &   \multicolumn{2}{c}{1.2s}&&\multicolumn{2}{c}{3.7s}&& \multicolumn{2}{c}{14.7s} \\[7pt]
        \multirow{5}{*}{Martingale posterior} 
  & $\mu_2$ & 92.8 & 6.1&& 93.5 & 2.7&& 94.7& 1.2\\ 
    & $s_2$ &  90.4 & 5.8&& 94.0& 2.7&& 94.5& 1.2\\ [5pt]
      &{Max. SE} & 0.4& 0.03  && 0.3 &0.01 &&   0.3 & 0.001\\[2pt]
         & Run-time &   \multicolumn{2}{c}{0.003s} && \multicolumn{2}{c}{0.003s}&& \multicolumn{2}{c}{0.003s} \vspace{2mm}\\
         \multicolumn{10}{l}{\footnotesize Cov, coverage;  SE, standard error. Coverage is in \%; length has been multiplied by 10. }\\
         \multicolumn{10}{l}{\footnotesize Run-time is per sample; results are over 500 and 5000 repeats for Bayes and the martingale}\\
         \multicolumn{10}{l}{\footnotesize posterior respectively.}
    \end{tabular}}
    \label{tablelabel}
    \label{tab2_supp}
\end{table}

\subsection{Real data analysis}\label{sec:app_rwd}
In this section, we include additional results and details for the real world data example. Table \ref{tab_cov} outlines the baseline covariates and variable types for the real data example. Figure \ref{fig:samp_aids} compares the posterior 95\%
probability contours of the coefficients for treatment arms (2) and (3) for the martingale posterior with the exact predictive resampling ($N = n+50000$), truncated predictive resampling ($N = n+100$) and the hybrid predictive resampling scheme. The 95\% contours are estimated using bivariate kernel density estimation.
We see that the truncated scheme is too narrow as before, with the exact and hybrid schemes agreeing on the contour and marginal plots.
\begin{table}
    \caption{Covariates for the AIDS Clinical Trials Group Study 175 dataset. }
    \center
    \small
{   \begin{tabular}{cl}
        Covariate &  Variable Type (Units/values) \vspace{1mm}\\
  \hspace{10mm}Age\hspace{10mm} & Continuous (Years) \hspace{10mm}\\
  Weight & Continuous (kg)\\
  Karnofsky Score &Continuous (0-100)\\
  CD4 count at BL&Continuous (cells/$\text{mm}^3$)\\
CD8 count at BL& Continuous (cells/$\text{mm}^3$)\vspace{1mm}\\
Sex & Binary (M/F)\\
Race & Binary (White/Non-white)\\
Homosexual activity & Binary (Y/N)\\
Symptomatic & Binary (Y/N)\\
Hemophilia & Binary (Y/N)\\
IV drug use history & Binary (Y/N)\\
Antiretroviral history & Binary (Y/N)\vspace{1mm}\\
Treatment Arm & Categorical (1-4) \vspace{5mm}\\
      \multicolumn{2}{l}{\footnotesize   BL, baseline; IV, intravenous. Outcome is CD4 count} \\
      \multicolumn{2}{l}{\footnotesize (continuous; cells/$\text{mm}^3$) at $20 \pm 5$ weeks. }
    \end{tabular}}
    \label{tab_cov}
\end{table}
  \begin{figure}
  \center
\includegraphics[width=0.55\textwidth]{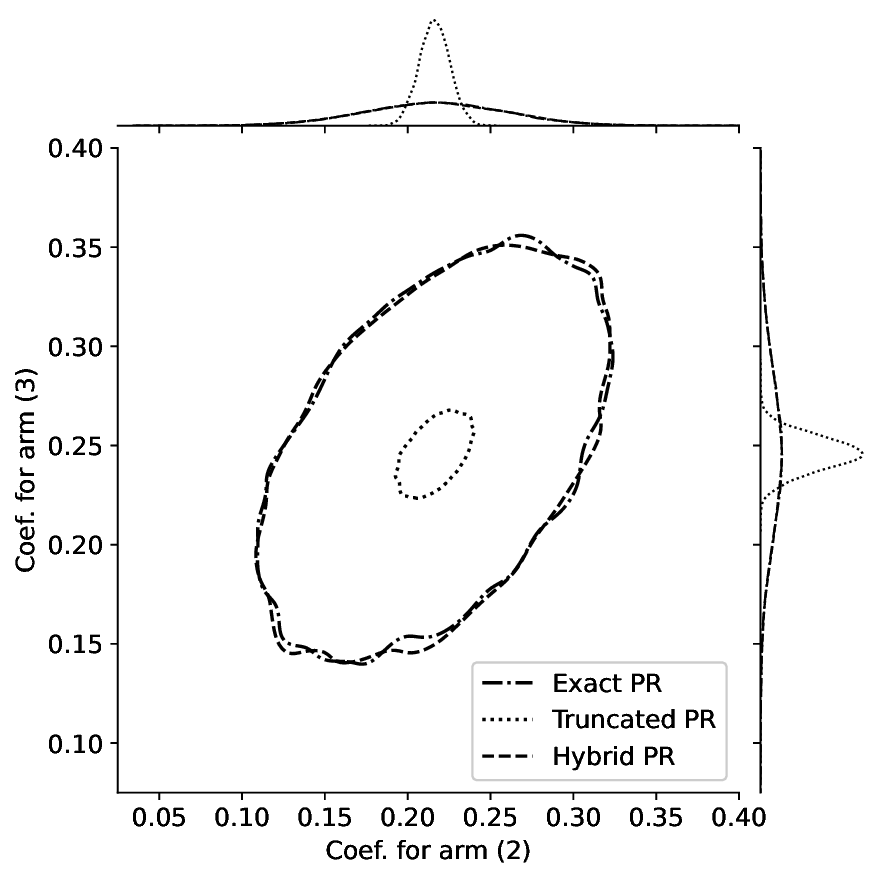}\vspace{2mm}
\caption{Posterior 95\% probability contour of the coefficients for treatment arms (2) and (3)
from the martingale posterior with exact sampling with $N = n + 50000$ (dot-dashes), hybrid sampling with $N = n+100$ (dashes) and truncation of $N = n+100$ (dot). }
\label{fig:samp_aids}
\end{figure}

\end{appendices}

\end{document}